\documentclass[a4paper]{article}

\usepackage{lmodern}
\usepackage[margin=1in]{geometry}
\usepackage{latexsym}
\usepackage{amssymb}
\usepackage{amsmath}
\usepackage{bm}
\usepackage{amsthm}
\usepackage{caption}
\usepackage{amsfonts}
\usepackage{algpseudocode}
\usepackage{dsfont}
\usepackage[title]{appendix}
\usepackage{color,xcolor}
\usepackage{longtable}
\usepackage{multirow}
\usepackage{booktabs}
\usepackage{makecell}
\usepackage{cases}
\usepackage{subcaption}

\usepackage{tcolorbox}
\graphicspath{{fig/}}
\tcbset{arc = 0pt, outer arc = 0pt,
 colback=white,
 boxsep = 0pt, left = 1pt, right = 1pt, top = 1pt, bottom = 1pt,
 boxrule = 0.5pt, bottomrule = .5pt, toprule = .5pt}

\newcommand{\RNum}[1]{\mathbf{\uppercase\expandafter{\romannumeral #1\relax}}}

\DeclareMathOperator*{\argmin}{arg\,min}

\newcommand{\RR}{\mathds{R}}

\newcommand{\lb}{\textsc{lb}}
\colorlet{blue}{blue!90!black}
\colorlet{red}{red!50!black}
\colorlet{green}{green!50!black}
\usepackage[pdfpagelabels,breaklinks,colorlinks,linkcolor=red,citecolor=green,urlcolor=red]{hyperref}
\usepackage{authblk}
\usepackage[linesnumbered,ruled,vlined]{algorithm2e}
\usepackage{enumitem}
\usepackage{cleveref}
\makeatletter
\newcommand{\refcheckize}[1]{
 \expandafter\let\csname @@\string#1\endcsname#1
 \expandafter\DeclareRobustCommand\csname relax\string#1\endcsname[1]{
 \csname @@\string#1\endcsname{##1}\wrtusdrf{##1}}
 \expandafter\let\expandafter#1\csname relax\string#1\endcsname
}
\makeatother

\setlist[description]{style=multiline,leftmargin=2.3em,topsep=3mm,itemsep=0mm}
\setlist[itemize]{style=standard,leftmargin=2em,topsep=3mm,parsep=0mm,itemsep=0mm}

\newcommand{\subsubsubsection}[1]{\paragraph{#1}\mbox{}}
\newcommand{\real}{\mathds{R}}

\newcommand{\nn}{\mathbb{N}}

\newcommand{\abs}[1]{\left \lvert #1 \right \rvert}

\newcommand{\kh}[1]{\left( #1 \right)}

\newcommand{\dist}{\mathrm{dist}}

\newcommand{\norm}[1]{\left \| #1 \right \|}

\newcommand{\divi}{\text{div}}
\newcommand{\dif}[1][]{\mathrm{d}#1}
\newcommand{\expect}[2]{\mathbb{E}_{#1}\left[{}#2\right]}
\newcommand{\ex}[1]{\mathbb{E}\left[{}#1\right]}
\newcommand{\proba}[1]{\mathbb{P}\left[{}#1\right]}
\newcommand{\field}[1]{\mathcal{{}#1}}

\newcommand{\SG}{{\rm SG}}

\newcommand{\lr}[3]{\left#1 {#2} \right#3}
\newcommand{\brc}[1]{\left( #1 \right)}

\newcommand{\lrangle}[1]{\lr{\langle}{#1}{\rangle}}

\DeclareMathOperator{\GND}{GND}

\newtheorem{theorem}{Theorem}[section]

\newtheorem{proposition}{Proposition}[section]
\newtheorem{lemma}{Lemma}[section]
\newtheorem{example}{Example}[section]
\newtheorem{definition}{Definition}[section]
\newtheorem{assumption}{Assumption}[section]

\newcounter{rmk}
\newtheorem{remark}[rmk]{Remark}

\numberwithin{equation}{section}

\title{\Large\bf A Stochastic Gradient Descent Method for Globally Minimizing Nearly Convex Functions\thanks{The research of this work was supported partly by the National Key R\&D Program of China [No. 2021YFA001300], the National Natural Science Foundation of China [Nos. 12271291, 12271150],
the Hunan Provincial Natural Science Foundation of China [No. 2023JJ10001],
and The Science and Technology Innovation Program of Hunan Province [No. 2022RC1190].} 
}
\date{\today}
\author{
Chenglong Bao\thanks{Yau Mathematical Sciences Center, Tsinghua University, Beijing, China, and Yanqi Lake Beijing Institute of Mathematical Sciences and Applications, Beijing, China (\url{clbao@mails.tsinghua.edu.cn}).}
\qquad
Liang Chen\thanks{School of Mathematics, Hunan University, Changsha, China (\url{chl@hnu.edu.cn}).}
\qquad
Weizhi Shao\thanks{Yau Mathematical Sciences Center, Tsinghua University, Beijing, China (\url{swz22@mails.tsinghua.edu.cn}).}
}

\begin{document}
\maketitle

\begin{abstract}
This paper proposes a stochastic gradient descent method with an adaptive Gaussian noise term for the global minimization of nearly convex functions, which are nonconvex and possess multiple strict local minimizers. The noise term, independent of the gradient, is determined by the difference between the current function value and a lower bound estimate of the optimal value. In both probability space and state space, we show that the proposed algorithm converges linearly to a neighborhood of the global optimal solution. The size of this neighborhood depends on the variance of the gradient and the deviation between the estimated lower bound and the optimal value. In particular, when full gradient information is available and a sharp lower bound of the objective function is provided, the algorithm achieves linear convergence to the global optimum. Furthermore, we introduce a double-loop scheme that alternately updates the lower bound estimate and the optimization sequence, enabling convergence to a neighborhood of the global optimum that depends solely on the gradient variance. Numerical experiments on several benchmark problems demonstrate the effectiveness of the proposed algorithm.

\bigskip\noindent
{\bf Keywords}:
stochastic gradient descent; 
global optimization;
linear convergence;
nearly convex functions

\bigskip\noindent
{\bf MSC2000 subject classification}: {65K05; 90C26; 90C06}
\end{abstract}

\section{Introduction}
Gradient descent (GD) is a classical method for solving the unconstrained nonlinear optimization problem
\[
\label{eq:prob}
\min_{x\in\real^d} f(x),
\]
where $\real^d$ is the $d$-dimensional Euclidean space and the objective function $f$ is differentiable but not necessarily convex.
Given an initial point $x_0\in\real^d$, it progressively calculates
$$
x_{t+1} : = x_t - \eta_t\nabla f(x_t),\quad  t=0,1,\ldots,
$$
until some stopping criterion is satisfied, where $\eta_t>0$ is the step size (or the learning rate in the machine learning community).

The convergence properties of the gradient descent (GD) method have been extensively studied when additional properties are imposed on $f$, such as convexity and Lipschitz continuity of the gradient. However, when the objective function $f$ is non-convex, the convergence analysis becomes challenging. To address this, several regularity conditions have been proposed, including the restricted secant inequality (RSI) \cite{zhang2013gradient}, the Polyak-Lojasiewicz (PL) condition \cite{polyak1963gradient}, and the quadratic growth (QG) condition \cite{necoara2019linear}. Despite these efforts, the main obstacle remains that GD can become trapped in stationary points.

To escape from non-optimal stationary points, a widely adopted strategy involves a noise term to inject random perturbations into the GD updates, i.e., setting
\begin{equation}\label{update:GD-noise}
x_{t+1} : = x_t - \eta_t \nabla f(x_t) + \sigma_t\xi_t,\quad  t=0,1,\ldots,
\end{equation}
where each $\xi_t$ is a random variable following a certain distribution, and $\sigma_t$ is the strength of the noise.
A well-known algorithm in the format of \eqref{update:GD-noise} is the stochastic gradient descent (SGD).
SGD and its variants have been widely used for training deep neural networks, especially when the full gradient is not accessible or too expensive to compute, e.g., when minimizing an expected risk given by
$f(x):=\expect{z}{F(x, z)}$, where $z$ is a random vector drawn from an unknown distribution. 
SGD updates the sequence $\{x_t\}$ iteratively using the rule:
\begin{equation}\label{update:SGD}
x_{t+1}: = x_t - \eta_t\SG(x_t), \quad  t=0,1,\ldots.
\end{equation}
Here, $\eta_t$ is usually called the step size and $\SG(x)$ is the stochastic gradient that satisfies $\ex{\SG(x)}=\nabla f(x)$. 
Hence by letting $\xi_t = \nabla f(x_t) - \SG(x_t)$, we can see that \eqref{update:SGD} is a special case of \eqref{update:GD-noise}. 
The step size $\eta_t$ controls not only the step size, but also the noise strength (like $\sigma_t$ in \eqref{update:GD-noise}), and hence should be carefully tuned in practice \cite{you2019does}.
Note that when using SGD to minimize the expected risk $\expect{z}{F(x, z)}$ one may consider using the empirical risk $\frac{1}{n}\sum_{i = 1}^{n} F(x, z_i)$ as a surrogate. 
When $F(x, z)$ is strongly convex with respect to $x$ and $n$ is sufficiently large, it has been proven in \cite{shalev2009stochastic,bottou2010large} that SGD with $\SG(x_t)=\nabla_xF(x, z_i)$ with $i$ being randomly chosen reaches an $\epsilon$-excess error (the error between the expected risk $f(x)$ and the optimal value $f^*$) with fewer gradient calls than GD in an asymptotic sense.
When $f$ is not convex, the noise induced by $\SG(x)$ is critical for the provable ability of SGD to escape from saddle points~\cite{ge2015escaping,jin2017escape} or sharp minima \cite{kleinberg2018alternative,loshchilov2016sgdr}.

Another commonly used choice of \eqref{update:GD-noise} is to set $\xi_t\sim\mathcal{N}\left(0,\frac{1}{d}I_d\right)$, that is, the standard $d$ multivariate Gaussian distribution with $0$ and $\frac{1}{d}I_d$ being the mean and the covariance matrix, respectively. Here, $I_d$ is the $d\times d$ identity matrix.
In this setting, finding an appropriate decay rate of the noise strength $\sigma_t$ is important for the convergence to the global minimum, as elaborated in \cite{chiang1987diffusion,gelfand1991recursive, geman1986diffusions,hwang1990large,kushner1987asymptotic}.
When the step size $\eta_t$ tends to zero, \eqref{update:GD-noise} approaches a continuous limit described by the stochastic differential equation
\begin{equation}\label{eq:diffusion}
    \dif x = -\nabla f(x) \dif[t] + \sigma(x, t)\dif[w],
\end{equation}
where $x:\real\to\real^d$ is an unknown differentiable function with respect to $t$, $w$ is a standard Brownian motion, and $\sigma$ is a real-valued function. The process $x$ in \eqref{eq:diffusion} is called the diffusion process, which has been used to study the behavior of the stochastic gradient descent (SGD) method.
In general, it has been shown in \cite{chiang1987diffusion, geman1986diffusions} that the critical decay rate of the noise strength $\sigma(x,t)$ is $\mathcal{O}(1/\sqrt{\ln t})$.
Here, the critical decay rate means that to ensure global convergence, the noise strength must not decay faster than $\mathcal{O}(1/\sqrt{\ln t})$; otherwise, the process may get trapped in local minima. 
Moreover, denote $G$ as the domain of attraction for a local minimum $\hat{x}$ (i.e., the largest local domain where the gradient flow converges to $\hat x$ \cite{dembo2009large}), 
fix $\sigma(x, t) \equiv \sigma$ in \eqref{eq:diffusion},
and let $x(0) \in G$ be the starting point.
Then, according to the large deviation principle~\cite{dembo2009large},
the stopping time $\tau_\sigma$ that the diffusion process reaches the boundary $\partial G$ of $G$ satisfies
$$
    \tau_\sigma \sim \exp\left(\frac{2}{\sigma^2} \inf_{x \in \partial G} (f(x) - f(\hat{x}))\right)\  \text{ as }\ \sigma \rightarrow 0^+,
$$
which indicates that the time it takes for the process $x$ to escape from a local minimum could be exponentially long. This formula motivates the adaptive setting of $\sigma$ to reduce the time to escape local minima to a constant.

The recent work \cite{engquist2022algebraically} proposes an adaptive scheme to adjust $\sigma_t$
in \eqref{update:GD-noise} that depends on $f(x_t)$, i.e.,  the objective value in the current iteration.
Specifically, given a monotonically decreasing sequence $\{\upsilon_t\}$, if $x_t$ lies in the sublevel set of $f$ determined by $\upsilon_t$, it accepts $x_{t+1}$ calculated by \eqref{update:GD-noise}. 
Otherwise, it randomly samples $x_{t+1}$ from the standard multivariate Gaussian distribution.
Under certain assumptions, an algebraic convergence rate in both probability space and state space has been established in \cite{engquist2022algebraically}, which substantially improves the logarithmic convergence rate in \cite{chiang1987diffusion, geman1986diffusions}.
More recently, a derivative-free global optimization method was proposed in \cite{engquist2023adaptive} with a similar algebraic convergence rate corresponding to $\eta_t\equiv 0$ but with a carefully chosen scheme of $\sigma_t$ in \eqref{update:GD-noise}. 
However, the iteration complexity of these methods remains unknown, which motivates us to identify a non-convex function class that can achieve these results.

Inspired by the inherent smoothing property of SGD (see details in Remark 1), we introduce the definition of nearly convex functions, which can be viewed as perturbations of one-point strongly convex functions \cite[Assumption 1]{kleinberg2018alternative}. It is shown that nearly convex functions lie between two quadratic functions, thus clearly satisfying the QG condition. When the global minimizer set is unique, the one-point strongly convex function is equivalent to the RSI property, while nearly convex is strictly weaker than this condition.
Additionally, the class of nearly convex functions includes functions that are not PL (see Example 2). Moreover, PL functions can have multiple global minimizers, while nearly convex functions have a unique global minimizer. It is worth mentioning that a nearly convex function can be nonconvex and contain infinitely many saddle points and local minima, as illustrated in \Cref{subsec:examples}, whereas PL functions do not contain saddle points or local minima. One may refer to \Cref{sec:comapare} for more details.

To minimize the nearly convex functions, we propose a new perturbed SGD scheme,  
given by 
\begin{equation}\label{update:sgd-noise}
    x_{t+1} : = x_t - \eta_t \SG(x_t) + \sigma_t\xi_t, \quad  t=0,1,\ldots, 
\end{equation}
where $\sigma_t := \sqrt{\eta_ts(f(x_t-\eta_t\SG(x_t))-f_{\lb})^+}$ and $s>0$ is a hyper-parameter (cf. \Cref{sec:main_results} for details), and $\xi_t\sim\mathcal{N}\left(0,\frac{1}{d}I_d\right)$. 
Here, $s$ is a parameter that controls the noise strength, and $f_{\lb}$ denotes a lower bound estimation of $f$. This noise term, which depends on the function value, adds more noise when the gap between the function value and the lower bound is large and less noise when the gap is small. 
If the objective function $f$ is nearly convex and $f_{\lb} \le f^*$ is a valid lower bound of $f$, we show that the update scheme \eqref{update:sgd-noise} will reach an $\epsilon$-neighborhood of global minimum point linearly in both probability space and state space, where the radius $\epsilon$ depends on the variance of the stochastic gradient $\SG(x)$ and the difference between $f^*$ and $f_{\lb}$.
More specifically, if the full gradient is used as the stochastic gradient, under the nearly convex condition, we prove that the proposed algorithm converges to a global minimizer at a linear rate by setting $f_{\lb} = f^*$ and appropriately choosing $\eta_t$ and $s$, where $f^*$ denotes the global minimum value. This result provides the first theoretical guarantee that noisy gradient descent can efficiently escape local minima and achieve global convergence.

Additionally, we introduce a double-loop version of  \eqref{update:sgd-noise} (including \eqref{update:GD-noise}) which adaptively estimates $f_{\lb}$ such that the gap between $f^*$ and $f_{\lb}$ decreases below any given positive threshold with high probability. 
In particular, for this double-loop version, the iterates with full gradient can find a point in the neighborhood with radius $\epsilon$ of the global minimum point of a nearly convex function in $\mathcal{O}(\ln^2(1/\epsilon))$ steps with high probability.
Our results demonstrate that a specific class of nonconvex functions, i.e., the nearly convex functions, can be efficiently and effectively solved using the function value-dependent noise term in the proposed SGD method.
Numerical experiments on several concrete problems demonstrate the effectiveness of the proposed algorithm and validate our theoretical findings.

The rest of this work is organized as follows. 
In \Cref{sec:notation}, we provide the notation and some preliminary results. 
In \Cref{sec:main_results} we present the main algorithms and their convergence results.
The proofs of the results in this paper are given in \Cref{sec:convergence_proof}, and the numerical experiments are presented in \Cref{sec:numerical_experiments}. 
Finally, we conclude this paper in \Cref{sec:concludsion}.

\section{Notation and preliminaries}
\label{sec:notation}
We use $\nn$ to denote the set of natural numbers, and $\nn^+$ to denote its subset of positive integers.
Let $\real^d$ with $d\in\nn^+$ be the $d$-dimensional Euclidean space endowed with the standard inner product $\langle\cdot, \cdot\rangle$ and its induced norm $\|\cdot\|$. 
Given a point $x\in\real^d$ and a set $S\subset \real^d$, we use $\dist(x,S):=\inf\limits_{y\in S}\|x-y\|$ to denote the distance from $x$ to $S$
and use $\Pi_S(x):=\argmin\limits_{y\in A}\|x-y\|$ to denote the projection of $x$ onto $S$, which is not necessarily nonempty nor single-valued. 
For a continuously differentiable function $f:\real^d\to\real$, we use $f'(x)$ and $\nabla f(x)$ to denote its Jacobian and gradient, respectively.

The notation $\xi \in \mathcal{N}\big(0, \frac{1}{d}I_d\big)$ means that $\xi$ is $d$-variate normally distributed with a zero mean and $\frac{1}{d}I_d$ being the corresponding covariance matrix, where $I_d$ is the $d\times d$ identity matrix. 
We use $\ex{\cdot}$ (or $\expect{\xi}{\cdot }$) to denote expectation (or expectation with respect to $\xi$) and use ``a.s.'' to abbreviate ``almost surely'', which means that the probability of a certain event is one. 
 
In this paper, we focus on the unconstrained optimization problem \eqref{eq:prob}. 
We always assume that the solution set $\cal S$ to \eqref{eq:prob} is non-empty, and use $f^*$ to denote the corresponding optimal function value.

\section{Nearly convex functions}
\label{sec:ncf}
In this part, we introduce a new regularity condition, termed nearly convex, along with a discussion of its key properties. In the subsequent sections, we mainly focus on objective functions that admit a unique global minimizer.

We first propose the definition of one-point strongly convex functions and nearly convex functions. 
\begin{definition}\label{def:one_point}
The set of one-point strongly convex functions with $\alpha>0$ at $\bar x \in \real^d$ for $s \in \real$ is defined by 
\begin{equation*}
    \mathcal{P}_{\alpha}(\bar x, s) = \{\bar{f} \in C^1(\real^d) \mid \bar{f}(\bar x) = s, \langle \nabla \bar{f}(x), x - \bar x \rangle \ge \alpha \norm{x - \bar x}^2,\forall x\in\real^d\}. 
\end{equation*}
\end{definition}
As we will show in \cref{sec:comapare}, one-point strong convexity is equivalent to the RSI condition when the global minimizer is unique. Given a continuously differentiable function $f$, we define $\beta(f, \bar{x}, \alpha)$ that measures the distance from $f$ to $\mathcal{P}_{\alpha}(\bar x, f(\bar x))$ as follows.
\begin{definition}
\label{def:beta}
Given $\alpha>0$, $\bar x \in \real^d$, and $f\in \mathcal{C}^1(\real^d)$. 
Define the constant
\begin{equation*}
\beta(f, \bar x, \alpha) := \inf_{\bar{f} \in \mathcal{P}_{\alpha}(\bar x, f(\bar x))}\sup_{x \ne \bar x} \frac{2\abs{f(x) - \bar{f}(x)}}{\norm{x - \bar x}^2}.
\end{equation*}
\end{definition}
We immediately obtain the following result from \Cref{def:beta}.
\begin{lemma}
$\beta(f, \bar x, \alpha) \le C$ is equivalent to the following statement: $\forall n \in \nn^+$, there exists $\bar{f}_n \in \mathcal{P}_{\alpha}(\bar x, f(\bar x))$ such that
    \begin{equation*}
        \abs{f(x) - \bar{f}_n(x)} \le \kh{\frac{C}{2} + \frac{1}{n}}\norm{x - \bar x}^2
        \quad \forall x \in \real^d.
    \end{equation*}
\end{lemma}
The above lemma provides an intuitive interpretation of $\beta(f, \bar x, \alpha)$. Specifically, it characterizes $f$ as a small perturbation of a one-point strongly convex function, and the perturbation error is bounded by $\frac{\beta(f, \bar x, \alpha)}{2}\norm{x-\bar x}^2$. 
When $\beta(f, \bar x, \alpha)$ is relatively small, we may say that $f$ is ``nearly one-point convex'', which motivates the following definition of nearly convex functions.

\begin{definition}
\label{ass:main}
Let $L\geq\alpha>0$.
We say a continuously differentiable function $f:\real^d\to\real$ is $(\alpha, L)$-nearly convex if the following three conditions hold simultaneously:

\begin{enumerate}
\item[{\bf (a)}] $f$ admits a unique global minimizer $x^*\in\real^d$;

\item[{\bf (b)}] the function $f$ satisfies
 \begin{equation}\label{eq:asp.L}
 \norm{\nabla f(x)} \le L\norm{x - x^*}\quad \forall x\in\real^d;
 \end{equation}

\item[{\bf (c)}] the distance $\beta(f, x^*, \alpha)$ satisfies
\begin{equation}\label{eq:asseq}
    \beta(f, x^*, \alpha) \le \frac{1}{4}\sqrt{\frac{\alpha^5}{dL^3}}.
\end{equation}
\end{enumerate}
\end{definition}

We first provide the following remark on the motivation for studying nearly convex functions and their essential properties.
\begin{remark}
\label{motivation}
Assume that the stochastic gradient oracle takes the form of $\SG(x) = \nabla f(x) + \omega$, where $\omega$ is a zero-mean random noise term that may depend on $x$, SGD~\eqref{update:SGD} can be rewritten as
\begin{equation*}
 x_{t + 1} := x_t - \eta_t\nabla f(x_t) - \eta_t \omega_t,\quad t=0,1,\ldots.
\end{equation*}
Here, $\omega_1, \omega_2, \cdots$ are independent.
Following from \cite{kleinberg2018alternative}, we let $y_t := x_t - \eta_t\nabla f(x_t)$, 
so that $x_{t + 1} = y_t - \eta_t \omega_t$. Then the above SGD iterations can be interpreted as updating $\{y_t\}$ via
\begin{equation}
\label{update:yt}
y_{t + 1}: = y_t - \eta_t \omega_t - \eta_{t+1}\nabla f(y_t - \eta_t \omega_t).
\end{equation}
Taking expectation with respect to $\omega_t$ in \eqref{update:yt}, one obtains
\begin{equation}\label{update:yt.expec}
\expect{\omega_t}{y_{t + 1}} =  y_t - \eta_{t+1} \nabla \expect{\omega_t}{f(y_t - \eta_t \omega_t)}, 
\end{equation}
where $\expect{\omega_t}{f(y_t - \eta_t \omega_t)} := \int f(y_t - \eta_t \omega) \dif[\mathbb{P}](\omega)$ can be viewed as the convolution of the function $f$ and the noise $\omega_t$. 
Denoting $\hat{f}(y):= \expect{\omega_t}{f(y - \eta_t \omega_t)}$. 
One has from \eqref{update:yt.expec} that $\expect{\omega_t}{y_{t + 1}}$ is obtained by performing a single GD step on $\hat{f}(y)$ from $y_t$. 
Since the convolution has a smoothing effect, it may smooth the sharp local minima, which intuitively explains how SGD helps escape from the local minima. 

Consequently, for the function $f(x):= \bar{f}(x) + \epsilon(x)$, where $\bar{f}(x)$ is one-point strongly convex and $\epsilon(x)$ is a small perturbation function, we conjecture that the convolution mentioned above will reduce the role of $\epsilon(x)$ and the SGD step is mainly induced by $\bar{f}(x)$. 
We will validate this idea in \Cref{ass:main} under the condition that the perturbation is not too large.
\end{remark}

In \Cref{ass:main}, the condition \eqref{eq:asp.L} is weaker than the standard $L$-smooth assumption, which requires $\nabla f$ to be $L$-Lipschitz continuous over the entire space. In contrast, \eqref{eq:asp.L} only assumes that $\nabla f$ is isolated calm at the single point $x^*$. 
The constant $\frac{1}{4}\sqrt{\frac{\alpha^5}{dL^3}}$ in \eqref{eq:asseq} is set for ease of the forthcoming convergence analysis of algorithms. This bound may be relaxed to a larger value. 
We should mention that nearly convex functions are not necessarily convex, and we can even find an example that has infinitely many strict local minima and saddle points (see \Cref{subsec:examples} for details). 
We now present several properties of $(\alpha, L)$-nearly convex functions, along with brief proofs.

\begin{proposition}
\label{lm:nearly_convex}
Let $f:\real^d\to\real$ be $(\alpha, L)$-nearly convex and $x^*$ be a global minimizer of $f$ with $f(x^*)=f^*$. Then the following statements are valid.
\begin{description}
\item[\bf (a)] For any function $\bar{f} \in \mathcal{P}_\alpha(x^*, f^*)$ , one has 
\begin{equation*}
    \bar{f}(x) - f^* \ge \frac{\alpha}{2}\norm{x - x^*}^2.
\end{equation*}
\item[\bf (b)] One has $ \beta := \beta(f, x^*, \alpha) < \alpha $ and 
\begin{equation}\label{eq:lm.nearly_convex}
\frac{\alpha - \beta}{2}\norm{x - x^*}^2 \le f(x) - f(x^*) \le \frac{L}{2}\norm{x - x^*}^2.
\end{equation}
Moreover, $x^*$ is the unique global minimizer of $f$.
\end{description}
\end{proposition}

\begin{proof}
{\bf (a)} 
Fix any $x \in \real^d$, and define the uni-variate function $g(t): = \bar{f}(x^* + t(x - x^*))$,
so that $g(0) = \bar{f}(x^*)$ and $g(1) = \bar{f}(x)$.
Since $\bar{f} \in \mathcal{P}_\alpha(x^*, f^*)$, by \cref{def:one_point} it satisfies
$\langle \nabla \bar{f}(x), x - x^* \rangle \ge \alpha \norm{x - x^*}^2$ for all $x\in\RR^d$. Then
\begin{equation*}
g'(t) = \langle \nabla \bar{f}(x^* + t(x - x^*)), x - x^*\rangle \ge \alpha t\norm{x - x^*}^2
\quad \forall x\in \real^d.
\end{equation*}
Therefore, it comes from $\bar{f}^*=f^*$ that 
\begin{equation*}
\bar{f}(x) = g(1) = g(0) + \int_0^1 g'(t)\dif t \ge f^* + \frac{\alpha}{2}\norm{x - x^*}^2 \quad \forall x\in \real^d.
\end{equation*}

{\bf (b)}
From \cref{def:beta}, we know that for any $n \in \nn^+$, there exists $\bar{f}_n \in \mathcal{P}_\alpha(x^*, f(x^*))$ satisfying that $\abs{f(x) - \bar{f}_n(x)} \le \kh{\frac{\beta}{2} + \frac{1}{n}}\norm{x - x^*}^2$ for all $x \in \real^d$.
From \textbf{(a)} we have
\begin{equation*}
    \bar{f}_n(x) - f(x^*) \ge \frac{\alpha}{2}\norm{x - x^*}^2\quad \forall x\in \real^d.
\end{equation*}
Thus, it holds that
\begin{equation*}
f(x) - f(x^*) = f(x) - \bar{f}_n(x) + \bar{f}_n(x) - f(x^*) \ge \kh{\frac{\alpha - \beta}{2} - \frac{1}{n}}\norm{x - x^*}^2
\quad \forall x\in \real^d.
\end{equation*}
We obtain the inequality on the left side using the limit with $n\rightarrow +\infty$.
Similarly, according to \eqref{eq:asp.L}, one has for any $x\in\real^d$,
\begin{equation*}
f(x) - f(x^*) 
= \int_0^1 \langle \nabla f(x^* + t(x - x^*)), x - x^*\rangle \dif t
\le \int_0^1 Lt\norm{x - x^*}^2 \dif t
= \frac{L}{2}\norm{x - x^*}^2.
\end{equation*}
Finally, since $\beta < \alpha$ by \eqref{eq:asseq}, one has $f(x) > f(x^*)$ for any $x \ne x^*$.
\end{proof}

\begin{proposition}\label{prop:local}
Let $f:\real^d\to\real$ be an $(\alpha, L)$-nearly convex function with global minimizer $x^*$. Given $\hat x \in \mathbb{R}^d, \hat x \ne x^*$, and let $G$ be any open set satisfying that $\hat x \in G$ and $x^* \notin G$. Then we have
\begin{equation}
    \inf_{x \in \partial G} (f(x) - f(\hat x)) < \frac{1}{4}\sqrt{\frac{\alpha^5}{dL^3}} \norm{\hat{x} - x^*}^2,
\end{equation}
where $\partial G$ denotes the boundary of $G$. 
\end{proposition}

\begin{proof}
Let $\bar G = G \cup \partial G$ be the closure of $G$. Since $\hat x \in G, x^* \notin G$, we deduce that the line segment connecting $x^*$ and $\hat x$ must intersect the boundary $\partial G$ at some point $y$. We can express it as $y = (1 - t_0)x^* + t_0\hat x$ for some $t_0 \in (0, 1)$. Therefore, we have
\begin{equation}\label{eq:prop.1}
    \inf_{x \in \partial G} (f(x) - f(\hat x)) \le f(y) - f(\hat x).
\end{equation}
On the other hand, by the $(\alpha, L)$-nearly convexity of $f$, there exists a sequence of functions $\bar f_n \in \mathcal{P}_{\alpha}(x^*, f(x^*))$ such that 
\begin{equation*}
    \abs{f(x) - \bar f_n(x)} \le \brc{\frac{1}{8}\sqrt{\frac{\alpha^5}{dL^3}}+\frac{1}{n}}\norm{x - x^*}^2 \quad \forall x \in \real^d.
\end{equation*}
It follows that
\begin{equation*}
    f(y) - f(\hat x) \le \bar f_n(y) - \bar f_n(\hat x) + \brc{\frac{1}{8}\sqrt{\frac{\alpha^5}{dL^3}} + \frac{1}{n}}\brc{\norm{y - x^*}^2 + \norm{\hat x - x^*}^2}.
\end{equation*}
Define $g_n(t) = \bar f_n(x^* + t(\hat x - x^*))$, so that
\begin{equation*}
    g_n'(t) = \lrangle{\nabla\bar f_n(x^* + t(\hat x - x^*)), \hat x - x^*} \ge \alpha t\norm{\hat x - x^*}^2,
\end{equation*}
where the inequality follows from \cref{def:one_point}. This shows that $g_n(t)$ is monotonically increasing on $t \in [0, 1]$. Since $y = (1 - t_0)x^* + t_0\hat x$, we have
\begin{equation*}
    \bar f_n(y) = g_n(t_0) \le g_n(1) = \bar f_n(\hat x),
\end{equation*}
and then
\begin{equation*}
    f(y) - f(\hat x) \le \brc{\frac{1}{8}\sqrt{\frac{\alpha^5}{dL^3}} + \frac{1}{n}}(1 + t_0^2)\norm{\hat x - x^*}^2.
\end{equation*}
Letting $n \to \infty$, we obtain
\begin{equation}\label{eq:prop.2}
    f(y) - f(\hat x) \le \frac{1}{8}\sqrt{\frac{\alpha^5}{dL^3}}(1 + t_0^2)\norm{\hat x - x^*}^2 < \frac{1}{4}\sqrt{\frac{\alpha^5}{dL^3}} \norm{\hat{x} - x^*}^2.
\end{equation}
Combining \cref{eq:prop.1,eq:prop.2}, we complete the proof.
\end{proof}

\begin{remark}
    Suppose $\hat x$ is a strict local minimizer of $f$, and let $G$ be its domain of attraction (the largest local domain where the gradient flow converges to $\hat x$).
    Since $\hat x$ is a strict local minimizer, $G$ is a nonempty open set with $\hat x \in G$. Note that $x^*$ is the global minimizer, it follows that $x^* \notin \bar G$.
    Then \Cref{prop:local} implies that the barrier associated with $\hat x$ must not exceed $\frac{1}{4}\sqrt{\frac{\alpha^5}{dL^3}} \norm{\hat{x} - x^*}^2$. 
\end{remark}

\Cref{lm:nearly_convex} and \Cref{prop:local} show the global and local structural properties of nearly convex functions, respectively. In particular, \Cref{lm:nearly_convex} shows that nearly convexity implies the quadratic growth condition. Moreover, the local structure described in \Cref{prop:local} intuitively illustrates the usefulness of this condition: by limiting the depth of local minima, it facilitates efficient escape from local minima.

\subsection{Comparison to existing regularity conditions}
\label{sec:comapare}
To analyze optimization algorithms for solving \eqref{eq:prob}, numerous regularity conditions have been proposed and extensively studied. The most popular conditions for first-order algorithms include the following.

\begin{definition}[Restricted Secant Inequality (RSI)]
There exists a constant $\mu_r > 0$ such that
\begin{equation}\label{eq:RSI}
\langle \nabla f(x), x - x_p^* \rangle \geq \mu_r \|x - x_p^*\|^2 = \mu_r \dist(x, \mathcal{S})^2 \quad
\forall x \in \real^d, x_p^* \in \Pi_{\cal S}(x).
\end{equation}
\end{definition}

\begin{definition}[Polyak–Łojasiewicz (PL) Condition]
There exists a constant $\mu_p > 0$ such that
\begin{equation}\label{eq:PL}
\|\nabla f(x)\|^2 \geq 2\mu_p (f(x) - f^*) \quad
\forall x \in \real^d.
\end{equation}
\end{definition}

\begin{definition}[Quadratic Growth (QG)]
There exists a constant $\mu_q > 0$ such that
\begin{equation}\label{eq:QG}
f(x) - f^* \geq \frac{\mu_q}{2} \|x - x_p^*\|^2 = \frac{\mu_q}{2} \dist(x, \mathcal{S})^2 \quad
\forall x \in \real^d, x_p^* \in \Pi_{\cal S}(x).
\end{equation}
\end{definition}

Note that all of the conditions in the above definitions are relaxations of the strong convexity.
Moreover, these conditions do not imply the convexity of the objective functions.
For both the RSI and the PL condition, they imply that every stationary point of $f$ must be a global minimizer. However, none of these conditions require the global solution set $\cal S$ to be a singleton.
Under the additional assumption that $f$ is $L$-smooth (i.e., $\nabla f(x)$ is $L$-Lipchitz continuous), \cite{polyak1963gradient} and \cite{guille2022gradient} showed that either the RSI or the PL condition independently guarantees the linear convergence of GD. 
More recently, the convergence of several SGD algorithms was obtained under either the RSI or the PL condition \cite{kleinberg2018alternative,bassily2018exponential}. 
Although the QG condition alone does not ensure global convergence, it often plays a crucial role in convergence analysis.

The relationship among these conditions also shows their distinct roles. Under the $L$-smooth assumption, we typically have that the RSI implies the PL condition, and the PL condition implies the QG condition; see e.g., \cite{karimi2016linear}. 
Furthermore, \cite{liao2024error} shows that if $f$ is $\rho$-weakly convex with $\rho < \mu_q$, i.e. $f(x) + \frac{\rho}{2}\norm{x - x^*}^2$ is convex, then the RSI, the PL condition, and the QG are equivalent.

\begin{figure}
\centering
\includegraphics[width=.4\textwidth]{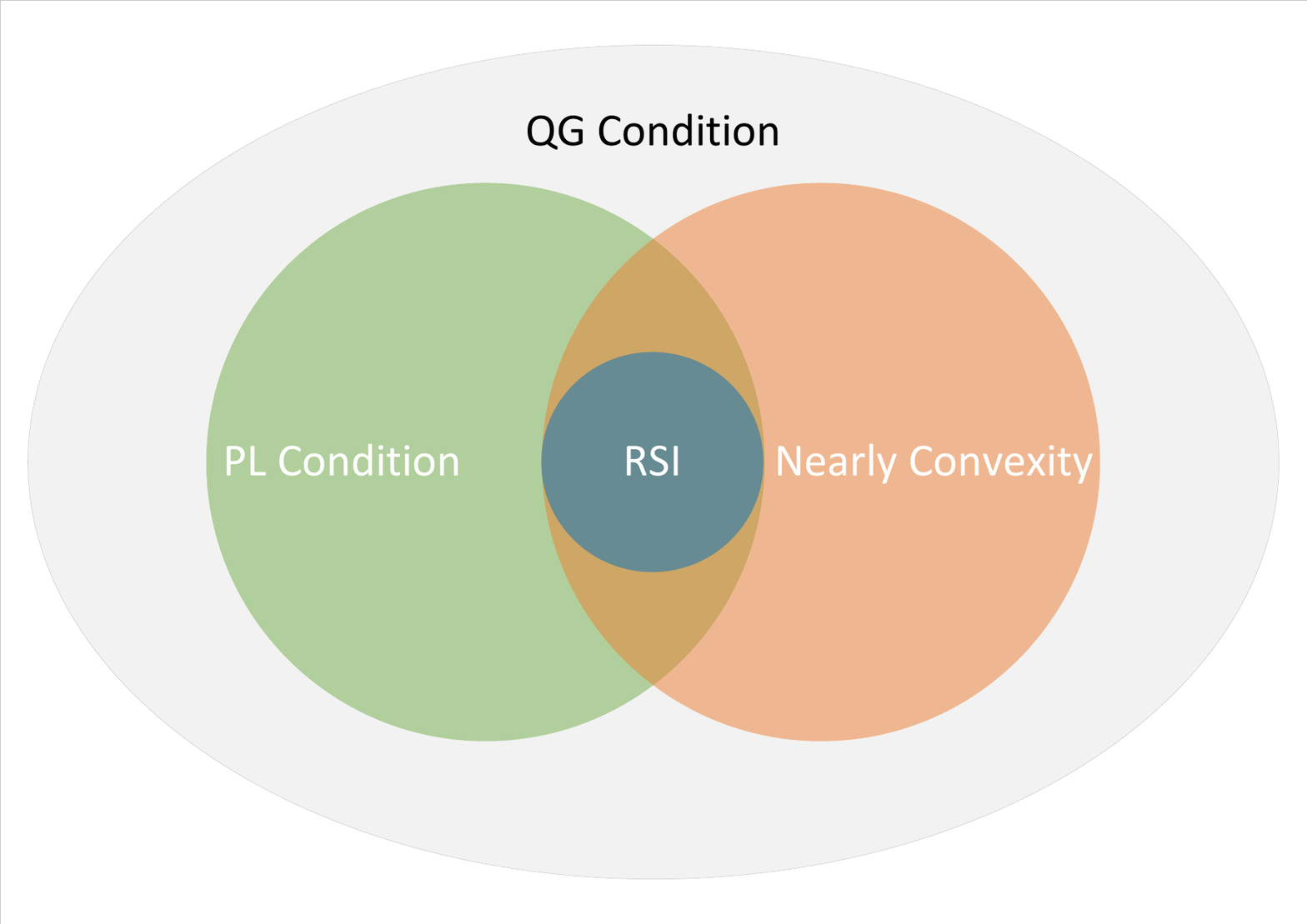}
\caption{The relationships among RSI \eqref{eq:RSI}, PL Condition \eqref{eq:PL}, Nearly Convexity (\Cref{ass:main}), and QG Condition \eqref{eq:QG}, under the assumption that $f$ is $L$-smooth and admits a unique global minimizer.}
\label{fig:relationships}
\end{figure}

It is of interest to make a comparison between nearly convexity and other regularity conditions introduced above. 
For simplicity, we assume that $f$ is $L$-smooth and admits a unique global minimizer $x^*$, and summarize their relationships in \Cref{fig:relationships}.
We first observe that the RSI condition \eqref{eq:RSI} is equivalent to the one-point strong convexity defined in \cref{def:one_point}, since we assume $\mathcal{S} = \{x^*\}$. This directly implies that RSI leads to nearly convexity.
On the other hand, the nearly convex property does not imply the PL condition.
This is because the PL condition requires all stationary points to be global optima, whereas nearly convex functions may admit strict local minima and saddle points.
However, whether the PL condition implies nearly convexity remains unclear. 
In particular, it is nontrivial to determine whether the PL condition is sufficient to guarantee \eqref{eq:asseq}. 
For an $L$-smooth function $f$ satisfying PL condition \eqref{eq:PL} with parameter $\mu_p > 0$, it holds that 
\begin{equation*}
    \frac{\mu_p}{2}\norm{x - x^*}^2 \le f(x) - f^* \le \frac{L}{2}\norm{x - x^*}^2\quad \forall x \in \real.
\end{equation*}
where the lower bound is obtained from \cite[Theorem 2]{karimi2016linear}, and the upper bound is a consequence of the $L$-smoothness. 
Now consider the function $\bar f(x) = f^* + \frac{\mu_p + L}{4}\norm{x - x^*}^2$, which satisfies RSI with parameter $\mu_r = \frac{\mu_p + L}{2}$.
Then we obtain that
$$
\abs{f(x) - \bar f(x)} \le \frac{L - \mu_p}{4}\norm{x - x^*}^2.
$$
This rough estimation yields only the bound
\begin{equation*}
    \beta\brc{f, \frac{\mu_p + L}{2}, f^*} \le \frac{L - \mu_p}{2},
\end{equation*}
which is far from the bound in \eqref{eq:asseq}. 
However, this does not imply that the PL condition cannot lead to nearly convexity, as a better choice of $\bar f$ may provide a sharper estimate. In particular, it can be shown that in the one-dimensional case, the PL condition does imply nearly convexity. In fact, in the one-dimensional case, the PL condition directly implies RSI. Specifically, in this case, if the PL condition holds, we have for all $x \in \real$ that 
\begin{equation*}
    \abs{f'(x)}^2 \ge 2\mu_p(f(x) - f^*) \ge \mu_p^2 \abs{x - x^*}^2
    \Rightarrow \abs{f'(x)} \ge \mu_p\abs{x - x^*}.
\end{equation*}
Moreover, the PL condition ensures that $x^*$ is the unique stationary point of $f$. Therefore, $f(x)$ must be strictly decreasing on $x < x^*$ and strictly increasing on $x > x^*$, which implies that
\begin{equation*}
    f'(x) \cdot (x - x^*) \ge \mu_p\abs{x - x^*}^2
    \quad \forall x \in \real.
\end{equation*}
This is exactly the RSI with parameter $\mu_r = \mu_p$. 
Since RSI implies nearly convexity, we conclude that the PL condition implies nearly convexity in the one-dimensional case.

\subsection{Examples of nearly convex functions}
\label{subsec:examples}
For ease of understanding, we provide some nontrivial examples of nearly convex functions. We begin with a simple subclass.

\begin{example}
\label{example:ncf.general}
Consider a function $f \in C^1(\real^d)$ that satisfies $\norm{\nabla f(x)} \le L\norm{x - x^*}$ and
\begin{equation*}
    \abs{f(x) - \frac{\alpha}{2}\norm{x - x^*}^2} \le \frac{1}{8}\sqrt{\frac{\alpha^5}{dL^3}}\norm{x - x^*}^2,
\end{equation*}
then we have $f$ is $(\alpha, L)$-nearly convex. This follows directly from the fact that $\frac{\alpha}{2}\norm{x - x^*}^2$ is one-point strongly convex.
\end{example}

It is worth mentioning that the functions in \Cref{example:ncf.general} can be nonconvex and may even admit strict local minima. 
In the following, we present two explicit examples that are nearly convex but not convex, both of which are special cases of \Cref{example:ncf.general}.

\begin{figure}
\centering
\includegraphics[width=.9\textwidth]{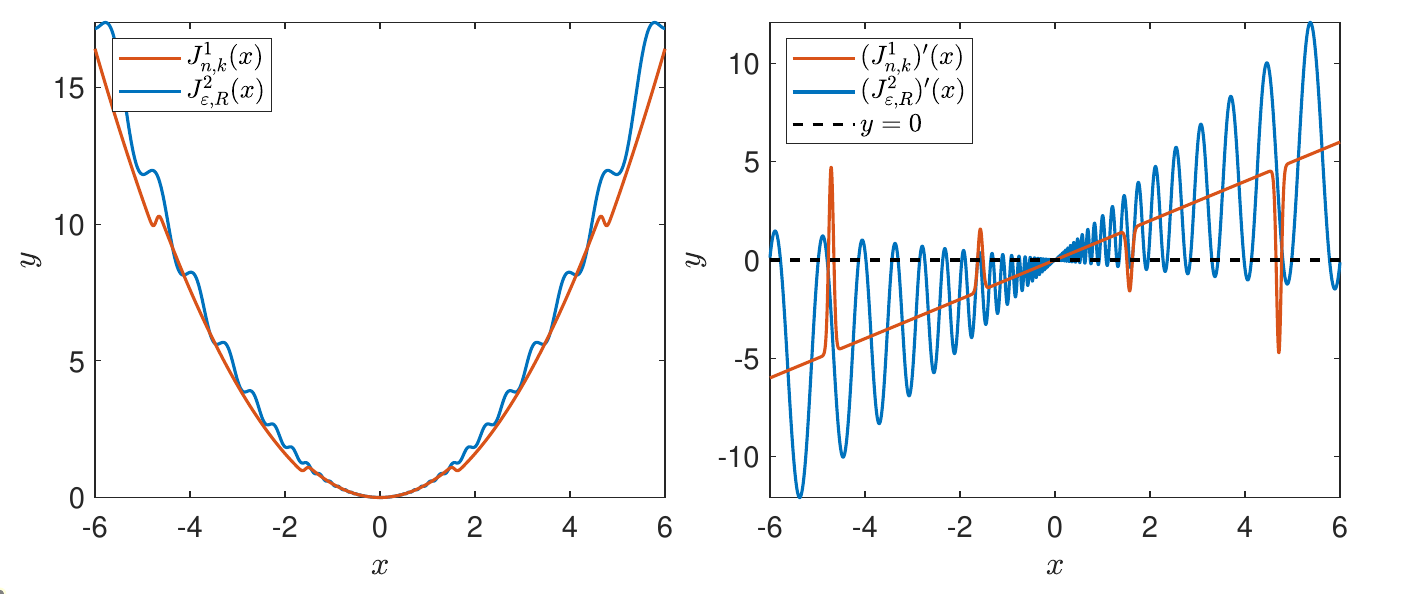}
\caption{Plots of the functions $J_{n, k}^1(x)$ and $J_{\varepsilon, R}^2(x)$ defined in \Cref{example:nearly_convex,example:sin_log}, together with their derivatives. 
For $J_{n, k}^1(x)$, we choose $k = 1$ and $n = n_1 = 199$. For $J_{\varepsilon, R}^2(x)$, we set $\varepsilon = \frac{2}{27}$ and $R = \frac{\sqrt{18161}}{8}$, such that $\varepsilon\sqrt{1 + R^2} = \frac{5}{4}$ and $4\varepsilon\brc{1 + \varepsilon\sqrt{1 + R^2}}^{3/2} = 1$. The two functions are nearly convex under these settings, as illustrated in \Cref{example:nearly_convex,example:sin_log}.}
\label{fig:ncvx_example}
\end{figure}

\begin{example}
\label{example:nearly_convex}
For $n, k\in \nn^+$, define the functions
\begin{equation*}
J_{n, k}^1(x) := \frac{1}{2}x^2 - \kh{1 + \frac{1}{k}}\int_{0}^{x} t(\sin t)^{2n} \dif t, \quad x\in\real.
\end{equation*}
Note that each $J_{n, k}^1$ is continuously differentiable with $(J_{n, k}^1)'(x) = x(1 - (1 + k^{-1}) (\sin x)^{2n})$. 
Moreover, the stationary points of $J_{n, k}^1$ are given by 
$x_j^\pm = j\pi \pm \arcsin{\kh{\frac{k}{k+1}}^{\frac{1}{2n}}},j\in\nn$, together with the origin.

Now, we show that for any $k \in \nn^+$, there exists $n_k \in \nn^+$, such that $J_{n, k}^1(x)$ is $(\frac{21}{25}, 1)$-nearly convex whenever $n \ge n_k$. 
It is easy to see that $\abs{(J_{n, k}^1)'(x)} \le   \abs{x - x^*}$, 
i.e., $(J_{n, k}^1)'(x)$ is Lipschitz continuous at $x = x^*$ with $L=1$ being the corresponding modulus. 
Define $g_n(x) := \int_{0}^{x} (\sin t)^{2n} \dif t$. 
For $x \in [0, \frac{\pi}{2}]$, 
one has $g_n''(x) = 2n(\sin x)^{2n-1}\cos x \ge 0$, 
which implies that $g$ is convex over $\left[0, \frac{\pi}{2}\right]$. 
Thus, one can get 
$$
\begin{array}{ll}
g_n(x) \le \frac{2}{\pi} g_n\kh{\frac{\pi}{2}}x \le \frac{2}{\pi}g_n(\pi)x, \quad\forall x \in \left[0, \frac{\pi}{2}\right],
\end{array}
$$
For~$x \in [\frac{\pi}{2}, \pi]$, one has $g_n(x) \le g_n(\pi) \le \frac{2}{\pi}g_n(\pi)x$. Then, for any $x\geq 0$, by writing $x=y+j\pi$ with $y\in[0,\pi]$ and $j\in\nn^+$, one can get
\begin{equation*}
\begin{array}{l}
g_n(x)=g_n(y+j\pi) = j g_n (\pi) + g_n(y) \le \frac{2}{\pi}g_n(\pi)(j\pi + y) = \frac{2}{\pi}g_n(\pi)x.
\end{array}
\end{equation*}   
Consequently, one has $\abs{g_n(x)} \le \frac{2}{\pi}g_n(\pi)\abs{x}$ for all $x \in \real$.
Note that for any $n\in \nn^+$ one has 
$$
g_n(\pi )=\int_{0}^{\pi } (\sin t)^{2n} \dif t 
=\frac{(2n-1)!!}{(2n)!!}\pi 
\quad
\mbox{with}
\quad 
\frac{(2n-1)!!}{(2n)!!}
:=\frac{2n-1}{2n}\times 
\frac{2n-3}{2n-2}\times 
\cdots
\times\frac{1}{2}. 
$$
We have $g_n(\pi)$ decreasing along with $n$ and converges to zero. 
Then for any $k \in \nn^+$, one can define $n_k$ by
\begin{equation}\label{def:nk}
    n_k := \min_{n} \left\{n \left|\frac{(2n-1)!!}{(2n)!!}\pi\le\frac{2k\pi}{25(k + 1)}\right.\right\}. 
\end{equation}
Moreover, one has $g_{n}(\pi )\le\frac{2k\pi}{25(k + 1)},\;\forall n\ge n_k$. 
When $n \ge n_k$, it holds that  
$$
0 \le \int_{0}^{x} t(\sin t)^{2n} \dif t
= \int_{0}^{\abs{x}} \abs{t}(\sin t)^{2n} \dif t
\le \abs{x} \abs{g_n(x)}
< \frac{2}{\pi}g_n(\pi)x^2 \le \frac{4k}{25(k + 1)}x^2.
$$
Hence one can get
$$
-\frac{4}{25}x^2 \le J_{n, k}^1(x) - \frac{1}{2}x^2 \le 0.
$$
As a consequence, letting $\alpha = \frac{21}{25}$, $L = d = 1$, and $x^* = 0$, we get
$$
\abs{J_{n, k}^1(x) - \frac{21}{50}x^2} \le \frac{2}{25}x^2
< \frac{\alpha}{8}\sqrt{\frac{\alpha^3}{dL^3}}\norm{x - x^*}^2.
$$
This validates that $J_{n, k}^1$ is a special case of \cref{example:ncf.general}, and is $(\frac{21}{25},1)$-nearly convex, whenever $n\ge n_k$. 
Moreover, from \eqref{def:nk}, one can explicitly calculate that  $n_1 = 199, n_2 = 112, n_3 = 89$, and  $n_4 = 78$. 
\end{example}

\begin{example}\label{example:sin_log}
For $\varepsilon \in (0, 1)$ and $R > 0$, consider the function
\begin{equation}\label{eq:sin_log}
J_{\varepsilon, R}^2(x) = 
\begin{cases}
    \dfrac{1 + \varepsilon\sin\kh{2R\log \abs{x}}}{2}x^2, & x \ne 0,\\
    0, & x = 0.
\end{cases}
\end{equation}
Obviously,  $J_{\varepsilon, R}^2(x)$ is continuously differentiable at $x = 0$ with $(J_{\varepsilon, R}^2)'(0) = 0$. And we observe that $\abs{J_{\varepsilon, R}^2(x) - \frac{1}{2}x^2} \le \frac{\varepsilon}{2}x^2$, this shows that $x^* = 0$ is the unique global minimizer of $J_{\varepsilon, R}^2$, and $\beta(J_{\varepsilon, R}^2, x^*, \alpha) \le \varepsilon$ with $\alpha = 1$. Next, we prove the following nearly convex properties of $J_{\varepsilon, R}^2$:
\begin{description}
    \item[$\bullet$] When $\varepsilon\sqrt{1 + R^2} < 1$, $J_{\varepsilon, R}^2$ is $(1 - \varepsilon\sqrt{1 + R^2}, 1 + \varepsilon\sqrt{1 + R^2})$-nearly convex.
    \item[$\bullet$] When $4\varepsilon\brc{1 + \varepsilon\sqrt{1 + R^2}}^{\frac{3}{2}} \le 1$, $J_{\varepsilon, R}^2$ is $(1, 1 + \varepsilon\sqrt{1 + R^2})$-nearly convex.
\end{description}
For the derivation of $J_{\varepsilon, R}^2$, when $x \ne 0$, we obtain
\begin{equation*}
\begin{aligned}
    (J_{\varepsilon, R}^2)'(x) &= (1 + \varepsilon\sin(2R\log \abs{x}) + \varepsilon R\cos(2R\log \abs{x})) \cdot x\\
    &= \kh{1 + \varepsilon\sqrt{1 + R^2}\sin(2R\log \abs{x} + \xi)}x,
\end{aligned}
\end{equation*}
where the angle $\xi$ satisfies 
\begin{equation*}
\sin \xi = \frac{R}{\sqrt{1 + R^2}},\quad \cos \xi = \frac{1}{\sqrt{1 + R^2}}.
\end{equation*}
This leads to
\begin{equation*}
\abs{(J_{\varepsilon, R}^2)'(x)} \le \kh{1 + \varepsilon\sqrt{1 + R^2}}\abs{x},
\end{equation*}
which shows that $f$ in \eqref{eq:sin_log} satisfies the condition \eqref{eq:asp.L} with $L = 1 + \varepsilon\sqrt{1 + R^2}$.
When $\varepsilon\sqrt{1 + R^2} < 1$, one can see that
\begin{equation*}
\lrangle{\nabla J_{\varepsilon, R}^2(x), x} \ge \kh{1 - \varepsilon\sqrt{1 + R^2}}x^2,
\end{equation*}
which shows that $J_{\varepsilon, R}^2 \in \mathcal{P}_{1 - \varepsilon\sqrt{1 + R^2}}(0, 0)$.
In this case, $\beta(J_{\varepsilon, R}^2, x^*, 1 - \varepsilon\sqrt{1 + R^2}) = 0$ and we have that $J_{\varepsilon, R}^2$ is $(1 - \varepsilon\sqrt{1 + R^2}, 1 + \varepsilon\sqrt{1 + R^2})$-nearly convex. 
For the case $4\varepsilon\brc{1 + \varepsilon\sqrt{1 + R^2}}^{\frac{3}{2}} \le 1$, we take $\alpha = 1$ and get $\beta(J_{\varepsilon, R}^2, x^*, \alpha) \le \varepsilon$. Hence the condition \eqref{eq:asseq} holds since 
\begin{equation*}
\beta(J_{\varepsilon, R}^2, x^*, \alpha) \le \varepsilon \le \dfrac{1}{4\brc{1 + \varepsilon\sqrt{1 + R^2}}^{\frac{3}{2}}}.
\end{equation*}
In this case, we obtain that $J_{\varepsilon, R}^2$ is $(1, 1 + \varepsilon\sqrt{1 + R^2})$-nearly convex.
\end{example}

\begin{table}[h]
\centering
\scriptsize
\begin{tabular}{c|cc}
\hline
\textbf{Conditions} & \textbf{Range} & \textbf{Parameter}\\
\hline
\\[-3pt]
SC & $\varepsilon\sqrt{1 + 5R^2 + 4R^4} < 1$ & $1 - \varepsilon\sqrt{1 + 5R^2 + 4R^4}$ \\[4pt]

RSI \eqref{eq:RSI} & $\varepsilon\sqrt{1 + R^2} < 1$ & $1 - \varepsilon\sqrt{1 + R^2}$ \\[4pt]

PL \eqref{eq:PL} & $\varepsilon\sqrt{1 + R^2} < 1$ & $\frac{\brc{1 - \varepsilon\sqrt{1 + R^2}}^2}{1 + \varepsilon}$ \\[4pt]

QG \eqref{eq:QG} & $\varepsilon < 1$ & $1 - \varepsilon$ \\[8pt]
\multirow{2}{*}{NC} & $\varepsilon\sqrt{1 + R^2} < 1$ & $(1 - \varepsilon\sqrt{1 + R^2}, 1 + \varepsilon\sqrt{1 + R^2})$ \\
& or $4\varepsilon\brc{1 + \varepsilon\sqrt{1 + R^2}}^{3/2} \le 1$ & or $(1, 1 + \varepsilon\sqrt{1 + R^2})$ 
\\
\hline
\end{tabular}
\caption{The ranges of $\varepsilon$ and $R$ for which \Cref{example:sin_log} satisfies various regularity conditions. The first column shows the condition names, where SC stands for strongly convex and NC stands for nearly convex. The Range column indicates the required ranges of $\varepsilon$ and $R$ for each condition to hold, while the Parameter column shows parameters of each condition when the ranges are satisfied.}
\label{tbl:sin_log}
\end{table}

Based on the discussion in \Cref{example:sin_log}, we summarize in \Cref{tbl:sin_log} the ranges of $\varepsilon$ and $R$ for which different regularity conditions hold. The table compares strongly convex (SC), RSI, PL condition, QG condition, and nearly convex (NC). Among these common conditions, excluding QG, NC has the broadest applicable range, as it holds either when $\varepsilon\sqrt{1 + R^2} < 1$ or when $4\varepsilon\brc{1 + \varepsilon\sqrt{1 + R^2}}^{3/2} \le 1$. In contrast, other conditions require more restrictive conditions like $\varepsilon\sqrt{1 + R^2} < 1$ or even $\varepsilon\sqrt{1 + 5R^2 + 4R^4} < 1$ for strong convexity. The QG condition is comparatively weak, holding for all $\varepsilon \in (0, 1)$ and $R > 0$; however, this also highlights its limitations, as QG alone does not ensure global convergence of GD or SGD. Detailed analyses for the conditions other than nearly convex are provided in \Cref{appendix:reg_cond.sin_log}.

\section{Stochastic gradient descent with Gaussian noise}
\label{sec:main_results}
In this part, we study stochastic gradient descent methods with adaptive Gaussian noise for minimizing nonconvex differentiable functions of the form \eqref{eq:prob}. Throughout, we typically adopt a constant step size (or learning rate), i.e., $\eta_t \equiv \eta$.
We begin by introducing and analyzing the Gaussian Noise Descent (GND) algorithm, presented in \Cref{alg:gngd}. Building upon this, we further develop a double-loop variant, the Double-Loop Gaussian Noise Descent (DL-GND) algorithm, as shown in \Cref{alg:dlgnd}.

\begin{algorithm}
 \caption{A Gaussian noise descent (GND) algorithm for solving \eqref{eq:prob}}
 \label{alg:gngd}
 \small
 \KwIn{
The initial point $x_0 \in \real^d$,
the stochastic gradient oracle $\SG(x)$,
the step size $\eta>0$,
the noise factor $s>0$,
a lower bound estimation $f_{\lb}$ of $f$, and the total number of iterations $T$.}
\KwOut{$x_{t_*}$.}
\For{$t = 0,\ldots,T - 1$}
{
$x_{t + \frac{1}{2}} \leftarrow x_{t} - \eta\SG(x_{t})$\;

$\sigma_t \leftarrow \sqrt{\eta s\kh{f\kh{x_{t + \frac{1}{2}}} - f_{\lb}}^+}$\;

$x_{t + 1} \leftarrow x_{t + \frac{1}{2}} - \sigma_t\xi_t$\
with $\xi_t \in \mathcal{N}\kh{0, \dfrac{1}{d}I_d}$\;
}
$t_* \leftarrow \min\left\{ \argmin\limits_{0 \le t \le T} f\kh{x_{t}}\right\}$.
\end{algorithm}

GND algorithm can be viewed as a variant of SGD. The stochastic oracle $\SG(x)$ in \Cref{alg:gngd} is a random variable with mean being the gradient $\nabla f(x)$, usually used to describe the general framework of SGD. For example, consider the expected risk $f(x) = \expect{z}{F(x, z)}$, we may denote $\SG(x) = \nabla_x F(x, z)$, where $z$ is a random variable drawn from a certain distribution.
In the convergence analysis, we make the following assumption on $\SG(x)$.

\begin{assumption}\label{asp:sg}
There exists a constant $r \ge 0$ such that the stochastic gradient oracle $\SG(x)$ satisfies
$$
 \ex{\SG(x)} = \nabla f(x)\quad\mbox{and}\quad \ex{\norm{\SG(x) - \nabla f(x)}^2} \le r^2.
$$
\end{assumption}

The key distinction between GND and standard SGD lies in the incorporation of an additional Gaussian noise term with adaptive variance following each SGD update. Specifically, the adaptive variance in \Cref{alg:gngd} is defined as $\sigma_t^2 = \eta s \brc{f(x_{t + \frac{1}{2}}) - f_{\lb}}^+$, where $f_{\lb}$ denotes an estimated lower bound of $f$, and $s > 0$ is a hyperparameter controlling the noise intensity. The guiding principle of \Cref{alg:gngd} is to introduce stronger noise when the current iterate is far from the global minimum and to reduce the noise level as the iterate approaches optimality. The quantity $\sqrt{\brc{f(x_{t + \frac{1}{2}}) - f_{\lb}}^+}$ serves as a heuristic measure of the distance to the global minimum. 

\begin{remark}\label{rmk:f_lb}
Note that $f_{\lb}$ is the input of \Cref{alg:gngd}, and we assume that $f_{\lb} \le f^*$ in the analysis of \Cref{alg:gngd}. This means that $\brc{f(x_{t + \frac{1}{2}}) - f_{\lb}}^+= f(x_{t + \frac{1}{2}}) - f_{\lb}$ for all $t$.
\end{remark}

We now present the global convergence analyses of Algorithm \ref{alg:gngd} for the case that $f$ is nearly convex and Assumption \ref{asp:sg} holds.
Specifically, we assume the objective function $f:\real^d\to\real$ in \eqref{eq:prob} is $(\alpha, L)$-nearly convex, and we adopt a step size $\eta < \frac{2\alpha}{L^2}$ in Algorithm \ref{alg:gngd}. For convenience, we define two constants $\lambda$ and $b$, which will be used throughout the analysis:
\begin{numcases}{}
\label{def:lambda}
\lambda := 2\alpha - \eta L^2 > 0,
\\
b := \frac{\eta r^2}{2\alpha - \eta L^2} + \frac{5\eta(2\alpha - \eta L^2) + 14}{42L}(f^* - f_{\lb}),
\label{def:b}
\end{numcases}
where $f^* := f(x^*)$ is the minimum value of $f$, $f_{\lb} \le f^*$ is a lower bound of $f$, and $r>0$ is the constant such that \Cref{asp:sg} holds.
For ease of presenting the most essential conclusion of our results, 
we take $\eta = \frac{2\alpha}{5L^2}, s = \frac{\lambda}{3L}$ in the convergence theorem, and one can obtain similar results by using similar analysis (see \cref{rmk:eta} for more details).

\begin{theorem}\label{thm:main}
Suppose that the objective function $f:\real^d\to\real$ in \eqref{eq:prob} is $(\alpha, L)$-nearly convex and that \Cref{asp:sg} holds. Define $\eta = \frac{2\alpha}{5L^2}, s = \frac{\lambda}{3L}$, and let $f_{\lb} \le f^*$ be any valid lower bound of $f$. Let $\lambda, b\ge 0$ be the constant defined by \eqref{def:lambda} and \eqref{def:b}. Let $\{x_t\}$ be the sequence generated by {\rm {\Cref{alg:gngd}}} with the specified $\eta, s, f_{\lb}$.
Define
$$
y_t: = x_t - \eta \nabla f(x_t),\quad t = 0, 1, \cdots, T.
$$
Then we have:
\begin{equation*}
    \ex{\norm{y_t - x^*}^2} \le \brc{1 - \frac{\eta\lambda}{100}}^t\norm{y_0 - x^*}^2 + 100b, \quad\forall\, t \in \{0, 1, \cdots, T\}.
\end{equation*}
Moreover, the following conclusions hold. 
\begin{enumerate}
\item [\bf (a)]
If $b = 0$,  for any $\varepsilon > 0$ one has $\norm{y_T - x^*}^2 \le \varepsilon$ with probability at least $1 - \zeta$, provided that $\zeta\in(0,1)$ and
$$
 T \ge \frac{100}{\eta(2\alpha - \eta L^2)}\ln\frac{\norm{y_0 - x^*}^2}{\zeta\varepsilon}.
$$
\item [\bf (b)]
If $b > 0$, there exists a certain $t \le T$ satisfying $\norm{y_t - x^*}^2 \le 200b$
with probability at least $1 - \zeta$,
provided that $\zeta\in(0,1)$ and
$$
 T \ge \frac{200}{\eta(2\alpha - \eta L^2)}\ln \frac{\norm{y_0 - x^*}^2}{100b\zeta}.
$$
Moreover, by letting  $\beta:=\beta(f, x^*, \alpha)<\alpha$, 
one has with probability at least $1 - \zeta$ that
\begin{numcases}{}
\label{eq:coro.2}
f(x_{t_*}) - f^* \le \frac{100Lb}{(1 - \eta L)^2},
\\
\label{result:coro.11}
\norm{x_{t_*} - x^*}^2 \le \frac{200Lb}{(1 - \eta L)^2(\alpha - \beta)},
\\
\label{result:coro.2}
\norm{y_{t_*} - x^*}^2 \le \kh{\frac{1 + \eta L}{1 - \eta L}}^2\frac{200Lb}{\alpha - \beta}.
\end{numcases}
\end{enumerate}
\end{theorem}

The proof of \Cref{thm:main} are postponed to \Cref{subsec:proof.thm1}.
Here we make the following remarks on \Cref{thm:main} and its proof that will be given in \Cref{subsec:proof.thm1}.

\begin{remark}\label{rmk:eta}
As will be observed from \eqref{eq:expg.final} in \Cref{subsec:proof.thm1} (the proof of \Cref{thm:main}), the global linear convergence of \Cref{alg:gngd} still holds for $\eta < \frac{2\alpha}{L^2}$ ,
provided that a constant $s>0$ can be found such that $\beta = \beta(f, x^*, \alpha)$ satisfies the following inequality:
\begin{equation}
\label{eq:eta.constrain}
\beta \sqrt{2d} \kh{\sqrt{\frac{2}{\pi\eta s(\alpha - \beta)}} + 1}\kh{1 + \eta sL}
 \le 2\alpha - \eta L^2 - \frac{sL}{2}.
\end{equation}
In our proof, we typically set $\eta = \frac{2\alpha}{5L^2}$ and $s= \frac{\lambda}{3L}$. In addition, \eqref{eq:eta.constrain} implies that $\eta$ cannot be arbitrarily small. However, when $f(x)$ is one-point strongly convex, meaning $\beta = 0$, we can arbitrarily choose $\eta \in (0, \frac{2\alpha}{L^2})$ by letting $s = 0$.
\end{remark}

\begin{remark}
If the full gradient is used and $f_{\lb} = f^*$, \Cref{asp:sg} holds with $r=0$, which implies $b=0$.  
In this case, \Cref{thm:main} shows that the sequence $\{y_t\}$ converges linearly to $x^*$. If $r>0$ and $f_{\lb} = f^*$, \Cref{thm:main} shows that $\{y_t\}$ converges linearly to a neighborhood of $x^*$, and the size of the neighborhood depends on the variance $r$ of the stochastic oracle. More specifically, if $f$ is one-point strongly convex at $x^*$, we can take $\eta=\epsilon$. Thus, with high probability, $\{y_t\}$ achieves a neighborhood $\mathcal{O}(\epsilon)$ of $x^*$ within $T = \mathcal{O}\kh{\frac{1}{\epsilon}\ln\frac{1}{\epsilon}}$ iterations.
\end{remark}

Note that, when implementing \Cref{alg:gngd}, it is generally impossible to evaluate $f^*$ as a sharp lower bound of $f$, but estimating an alternative lower bound can be much easier.
For example, many loss functions used in machine learning usually admit $0$ as a trivial lower bound.
Based on such an observation, we propose a double-loop version of \Cref{alg:gngd}. For convenience, we abbreviate Algorithm \ref{alg:gngd} to 
$$
x_{t_*} \leftarrow \GND(\SG, x_0, \eta, s, f_{\lb}, T),
$$ 
where the inputs can be viewed as the variables of a function named GND, and the unique output can be viewed as the function value. 
The double-loop variant of \Cref{alg:gngd} is given as follows.

\begin{algorithm}
 \caption{A double-loop Gaussian noise descent (DL-GND) algorithm for solving \eqref{eq:prob}}
 \label{alg:dlgnd}
 \small
 \KwIn{
The initial point $x_0 \in \real^d$,
the stochastic gradient oracle $\SG(x)$,
the step size $\eta>0$,
the noise factor $s>0$,
a lower bound $f_{\lb}^0$ of $f$, 
a combination coefficient $\gamma \in (0, 1)$, 
a number $N$ of the outer loops, 
and the two numbers $T_1$ and $T_2$ for the total iterations in the first and second stages, respectively.}
\KwOut{$x_{\min}^N$.}
$x_{\min}^0 \leftarrow \GND(\SG, x_0, \eta, s, f_{\lb}^0, T_1)$ by implementing \Cref{alg:gngd}; 

 \For{$\nu = 0$ to $\nu = N - 1$}
 {
 $f_{\lb}^{\nu + 1} \leftarrow (1 - \gamma)f_{\lb}^\nu + \gamma f\kh{x_{\min}^\nu}$\;

 $x_{\min}^{\nu + 1} \leftarrow \GND(\SG, x_{\min}^\nu, \eta, s, f_{\lb}^{\nu + 1}, T_2)$ by implementing \Cref{alg:gngd}. 
 }
\end{algorithm}

In \Cref{alg:dlgnd}, the outer loop primarily updates the lower bound estimate of the objective function, while the inner loop follows the procedure of \Cref{alg:gngd}. Moreover, we establish the following result regarding the convergence properties of \Cref{alg:dlgnd}.

\begin{theorem}\label{thm:double_loop}
Suppose that the objective function $f:\real^d\to\real$ in \eqref{eq:prob} is $(\alpha, L)$-nearly convex and \Cref{asp:sg} holds.
Let $\eta = \frac{2\alpha}{5L^2}, s = \frac{\lambda}{3L}$ and assume that $f_{\lb}^0 < f^*$. 
For any $\varepsilon > 0$ and  $\zeta \in (0, 1)$, define
\begin{equation}
\label{eq:defgamma}
\begin{array}{ll}
b^0 := \frac{\eta r^2}{\lambda} + \frac{5\eta\lambda + 14}{42L}(f^* - f_{\lb}^0),\quad
b^\varepsilon: = \frac{\eta r^2}{\lambda} + \frac{5\eta\lambda + 14}{42L}  \varepsilon,
\quad\mbox{and}\quad
\gamma: = \frac{(1 - \eta L)^2\varepsilon}{(1 - \eta L)^2\varepsilon + 100Lb^\varepsilon}.
\end{array}
\end{equation} 
where $\lambda$ is defined in \eqref{def:lambda}. Denote $\beta = \beta(f, x^*, \alpha)$.
Suppose that the positive integers $N$, $T_1$, and $T_2$ in \Cref{alg:dlgnd} are chosen such that
\begin{equation}
\label{eq:t1t2}
\begin{array}{ll}
N \ge \frac{1}{\gamma}\ln\kh{(1 - \gamma) + \gamma\cdot\frac{f^* - f_{\lb}^0}{\varepsilon}}, 
\quad 
T_1 \ge \frac{200}{\eta\lambda}\ln \frac{\norm{y_0 - x^*}^2}{100b^0\zeta'},
\quad \mbox{and}\quad 
T_2 \ge \frac{200}{\eta\lambda} \ln \frac{2(1 + \eta L)^2Lb^0}{(1 - \eta L)^2(\alpha - \beta)b^\varepsilon\zeta'} 
\end{array}
\end{equation}
with $\zeta':= \frac{\zeta}{N+1}$.
Then, with probability at least $1 - \zeta$, one has that the output 
$x^N_{\min}$ of \Cref{alg:dlgnd} satisfies
\begin{equation}\label{eq:dlgnd.conv.bound}
    f\kh{x_{\min}^N} - f^* \le \frac{100Lb^\varepsilon}{(1 - \eta L)^2} 
    \quad \text{and}\quad
    \norm{x_{\min}^N - x^*}^2 \le \frac{200Lb^\varepsilon}{(1 - \eta L)^2(\alpha - \beta)}.
\end{equation}
\end{theorem}

The proof of \Cref{thm:double_loop} will be given in \Cref{subsec:proof.thm2}.
Here we make some remarks on \Cref{thm:double_loop}.

\begin{remark}
In \Cref{alg:dlgnd}, it is possible that $f_{\lb}^\nu > f^*$ occurs at some iteration $\nu$. In such cases, for any $i \ge \nu$, the behavior of the iterates is no longer covered by the guarantees of \Cref{thm:main} (although this does not invalidate the algorithm, as discussed in \cref{rmk:f_lb}). Nevertheless, by the definition of $x_{\min}^\nu$, we observe that the sequence $\{f(x_{\min}^\nu)\}$ is non-increasing. Therefore, even if $f_{\lb}^\nu > f^*$ occurs at some iteration, the final value $f(x_{\min}^N)$ will not increase thereafter. In particular, we analyze the event that $f(x_{\min}^\nu) - f^*$ falls below a prescribed threshold before the condition $f_{\lb}^\nu > f^*$ is triggered, and we show that this desirable event occurs with probability at least $1 - \zeta$ (see \Cref{subsec:proof.thm2} for details).

\end{remark}

\begin{remark}
The total number of iterations required in \Cref{thm:double_loop} to achieve the convergence bound in \eqref{eq:dlgnd.conv.bound} is estimated as 
\begin{equation*}
T_1 + NT_2 = \begin{cases}
\mathcal{O}\kh{\ln^2 \frac{1}{\varepsilon}} &\mbox{ if }r = 0,\\
\mathcal{O}\kh{\frac{1}{\varepsilon}\ln \frac{1}{\varepsilon}} &\mbox{ if }r > 0.
\end{cases}
\end{equation*}
Therefore, when employing the full gradient ($r = 0$) in \Cref{alg:dlgnd}, which can be interpreted as a noise-perturbed gradient descent method, the algorithm is guaranteed to find a point within an $\varepsilon$-neighborhood of the global minimizer $x^*$ in at most $\mathcal{O}\kh{\ln^2 \frac{1}{\varepsilon}}$ steps with high probability.
\end{remark}

\section{Proofs of the main results}
\label{sec:convergence_proof}
This section first presents preliminary results related to the probability theory and stochastic process and then provides the proofs for \Cref{thm:main,thm:double_loop}.

\subsection{Preliminaries}

Recall that the Gamma function on positive real numbers is defined by
$$
\Gamma(z):=\int_0^\infty t^{z-1}e^t\dif t,\quad 0<z\in\real. 
$$
The following Lemma is a basic result of probability theory.
\begin{lemma}
\label{lem:intall}
Let $\xi \sim \field{N}(0, \frac{1}{d} I_d)$ be a random vector in $\real^d$ and $\Gamma(\cdot)$ be the Gamma function.
One has that
\begin{equation}
\label{eq:int.1234}
\begin{array}{ll}
\ex{\norm{\xi}} = \frac{\sqrt{2}}{\sqrt{d}}\frac{\Gamma\kh{\frac{d + 1}{2}}}{\Gamma\kh{\frac{d}{2}}},\quad
\ex{\norm{\xi}^2} = 1,\quad
\ex{\norm{\xi}^3} = \frac{\sqrt{2}}{\sqrt{d}}\frac{\Gamma\kh{\frac{d + 1}{2}}}{\Gamma\kh{\frac{d}{2}}}\kh{1 + \frac{1}{d}}
\ \,
\mbox{and}
\quad 
\ex{\norm{\xi}^4} = 1 + \frac{2}{d}.
\end{array}
\end{equation}
\end{lemma}
\begin{proof}
Note that $\|\sqrt{d}\xi\|$ is a random variable following a $\chi$-distribution with $d$ degrees of freedom. One can readily get \eqref{eq:int.1234} by examining the raw moments of the $\chi$-distribution.
\end{proof}

Based on the above Lemma, we have the following preliminary result, whose proof is given in \Cref{sec:proof1}.

\begin{lemma}
\label{lm:eps_bound_new}
Let $\varepsilon:\real^d\to\real$  be a continuously differentiable function such
that $\abs{\varepsilon(x)} \le \frac{l }{2} \norm{x - x^*}^2$ with $l >0$ and $x^*\in\real^d$.
Given any $x \in \real^d$ and $\delta > 0$, one has for $\xi \sim \field{N}(0, \frac{1}{d} I_d)$ that
\begin{numcases}{}
\label{eq:eps_bound.1}
\abs{\expect{\xi}{\langle \nabla \varepsilon(x - \delta\xi), x - x^* \rangle}} \le l  \sqrt{\frac{d}{2}}\frac{\norm{x - x^*}}{\delta\sqrt{\pi}}\kh{\norm{x - x^*}^2 + \delta^2\kh{1 + \frac{1}{d}}},
\\[1mm]
\label{eq:eps_bound.2}
 \abs{\expect{\xi}{\langle \nabla \varepsilon(x - \delta\xi), \delta\xi \rangle}} \le l  \sqrt{\frac{d}{2}}\kh{\norm{x - x^*}^2 + \delta^2\kh{1 + \frac{1}{d}}},
\\[1mm]
\label{result:eps_bound}
 \abs{\expect{\xi}{\langle \nabla \varepsilon(x - \delta\xi), x - \delta\xi - x^* \rangle}}
 \le l  \sqrt{\frac{d}{2}} \kh{\frac{\norm{x - x^*}}{\delta\sqrt{\pi}} + 1}
 \kh{\norm{x - x^*}^2 +  \delta^2\kh{1 + \frac{1}{d}}}.
\end{numcases}
\end{lemma} 

Note that \Cref{lm:eps_bound_new} contains the specific results of the convolution effect discussed in \Cref{motivation}.
The following lemma on stochastic processes with its proof given in \Cref{sec:proof2} will be used in the convergence analysis of the algorithm proposed in the next section.

\begin{lemma}\label{lm:st}
Let $(\Omega, {\mathcal F}, {\mathbb P})$ be a probability space, and
$\{X_t\}_{t\in {\mathds N}}$ be a stochastic process such that each $X_t$ is real-valued and integrable.
Denote $(\mathcal F_t)_{t \in {\mathds N}}$ be the generated filtration of $\{X_t\}$ so that each $X_t$ is measurable with respect to the $\sigma$-field $\field{F}_t$.
Suppose that there are constants $\theta \in (0, 1)$ and $b > 0$ such that
$$
\ex{X_{t + 1} \mid \field{F}_t} \le \theta X_t
\quad\mbox{and}\quad X_t \ge -b\ \text{a.s.}.
$$
Then, for any given nonnegative integer $M$ and real number $\ell>0$,
\begin{equation}
\label{def:st.B}
\proba{\exists t \le M, X_t < \ell}\ge 1 - \kh{\frac{b + \theta\ell}{b + \ell}}^M\frac{B}{\ell}
\quad\mbox{with}\quad
B: = \ex{X_0 \mathds{1}\{X_0 \ge \ell\}} \ge 0.
\end{equation}
\end{lemma}

\subsection{Proof of Theorem \ref{thm:main}}\label{subsec:proof.thm1}
The following Lemma plays a key role in the convergence analysis.
\begin{lemma}\label{lm:key}
Suppose that the function $f:\real^d\to\real$ in \eqref{eq:prob} is $(\alpha, L)$-nearly convex with $x^*\in\real^d$ being one of its global minimizers.
Define $f^*=f(x^*)$ and let $f_{\lb}\le f^*$ be a lower bound estimation of $f$.  
Let $\lambda>0$ and $b\ge 0$ be the two constants defined by \eqref{def:b} with $\eta := \frac{2\alpha}{5L^2}$.
Define the constant $s: = \frac{\lambda}{3L}$ and the functions
$$
G(x): = x - \eta\nabla f(x) \quad\mbox{and}\quad
\sigma(x):= \sqrt{\eta s\kh{f(x) - f_{\lb}}},
\quad x\in\real^d, 
$$
Then, one has with $\xi \sim \mathcal{N}(0, \frac{1}{d}I_d)$ that 
\begin{equation}
\label{result:expect_decent}
\expect{\xi}{\norm{G(x - \sigma(x)\xi) - x^*}^2}
\le
\kh{1 - \frac{\eta\lambda}{100}}\norm{x - x^*}^2
+ \frac{\eta\lambda(5\eta\lambda + 14)}{42L}(f^* - f_{\lb})
\quad
\forall x\in\real^d.
\end{equation}
\end{lemma}
\begin{proof}
Note that
\begin{equation}\label{eq:pf_key:main}
\begin{aligned}
&\norm{G(x - \sigma(x)\xi) - x^*}^2
=
\norm{x - \sigma(x)\xi - \eta\nabla f(x - \sigma(x)\xi)-x^*}^2
\\
&=
\norm{x - \sigma(x)\xi - x^*}^2
+ \norm{\eta\nabla f(x -\sigma(x)\xi)}^2
- 2\eta \langle \nabla f(x - \sigma(x)\xi), x - \sigma(x)\xi - x^* \rangle\\
&\le (1 + \eta^2L^2)\norm{x - \sigma(x)\xi - x^*}^2
- 2\eta \langle \nabla f(x - \sigma(x)\xi), x - \sigma(x)\xi - x^* \rangle,
\end{aligned}
\end{equation}
where the last inequality comes from \eqref{eq:asp.L} of  \Cref{ass:main}.
By \cref{def:beta} we know that for any $n \in \nn^+$, there exists $\bar{f}_n \in \mathcal{P}_\alpha(x^*, f(x^*))$ satisfying that $\abs{f(x) - \bar{f}_n(x)} \le \kh{\frac{\beta(f)}{2} + \frac{1}{n}}\norm{x - x^*}^2, \forall x \in \real^d$. We use $\beta = \beta(f)$ in the following.
Denote $\epsilon_n = f - \bar{f}_n$, then we have $\abs{\epsilon_n(x)} \le \kh{\frac{\beta}{2} + \frac{1}{n}}\norm{x - x^*}^2$. From \cref{def:one_point},
$\langle \nabla \bar{f}_n(x), x - x^* \rangle \ge \alpha \norm{x - x^*}^2$ for all $x\in\RR^d$. Therefore,
\begin{equation}\label{eq:pf_key:inner}
\begin{aligned}
& -2\eta \langle \nabla f(x - \sigma(x)\xi), x - \sigma(x)\xi - x^* \rangle\\
&= - 2\eta \langle \nabla \bar{f}_n(x- \sigma(x)\xi), x - \sigma(x)\xi - x^* \rangle
- 2\eta \langle \nabla \epsilon_n(x - \sigma(x)\xi), x - \sigma(x)\xi - x^* \rangle\\
&\le - 2\eta\alpha\norm{x - \sigma(x)\xi - x^*}^2
- 2\eta \langle \nabla \epsilon_n(x - \sigma(x)\xi), x - \sigma(x)\xi - x^* \rangle.
\end{aligned}
\end{equation}
Recall from \eqref{def:lambda} that $\lambda = 2\alpha - \eta L^2 > 0$.
Combining \eqref{eq:pf_key:main} and \eqref{eq:pf_key:inner} obtains
$$
\norm{G(x - \sigma(x)\xi) - x^*}^2
\le (1-\eta\lambda)\norm{x - \sigma(x)\xi - x^*}^2
- 2\eta \langle \nabla \epsilon_n(x - \sigma(x)\xi), x - \sigma(x)\xi - x^* \rangle.
$$
This inequality holds for any $n \in \nn^*$.
Therefore, by taking expectation with respect to $\xi$ and using \eqref{eq:int.1234} of \Cref{lem:intall}, one gets
\[
\label{eq:expg}
\begin{aligned}
&\expect{\xi}{\norm{G(x + \sigma(x)\xi) - x^*}^2} 
\\
&\le (1 - \eta\lambda)\expect{\xi}{\norm{x + \sigma(x)\xi - x^*}^2}
+ 2\eta\abs{\expect{\xi}{\langle \nabla \epsilon_n(x - \sigma(x)\xi), x - \sigma(x)\xi - x^* \rangle}}\\
&= (1 - \eta\lambda)\kh{\norm{x - x^*}^2 + \sigma(x)^2}
+ 2\eta\abs{\expect{\xi}{\langle \nabla \epsilon_n(x - \sigma(x)\xi), x - \sigma(x)\xi - x^* \rangle}}.
\end{aligned}
\]

Note we have $\abs{\epsilon_n(x)} \le \kh{\frac{\beta}{2} + \frac{1}{n}}\norm{x - x^*}^2$. Therefore, by taking $\delta=\sigma(x)$ in \eqref{result:eps_bound} of \Cref{lm:eps_bound_new}, one gets
\begin{equation}
\label{eq:estimate}
\begin{aligned}
&\abs{\expect{\xi}{\langle \nabla \epsilon_n(x - \sigma(x)\xi), x - \sigma(x)\xi - x^* \rangle}}
\\&
\le \kh{\beta + \frac{2}{n}} \sqrt{\frac{d}{2}}
\kh{\frac{\norm{x - x^*}}{\sqrt{\pi}\sigma(x)} + 1}
\kh{
\norm{x - x^*}^2 + 2\sigma(x)^2
}.
\end{aligned}
\end{equation}
Combine \eqref{eq:expg} and \eqref{eq:estimate} gives
\begin{equation*}
\begin{aligned}
&\expect{\xi}{\norm{G(x + \sigma(x)\xi) - x^*}^2} 
\\
&\le (1 - \eta\lambda)\kh{\norm{x - x^*}^2 + \sigma(x)^2}
+ 2\eta\kh{\beta + \frac{2}{n}} \sqrt{\frac{d}{2}}
\kh{\frac{\norm{x - x^*}}{\sqrt{\pi}\sigma(x)} + 1}
\kh{
\norm{x - x^*}^2 + 2\sigma(x)^2
}.
\end{aligned}
\end{equation*}
Taking $n \rightarrow +\infty$ we get
\begin{equation}
\label{eq:expg.1}
\begin{aligned}
&\expect{\xi}{\norm{G(x + \sigma(x)\xi) - x^*}^2} 
\\
&\le (1 - \eta\lambda)\kh{\norm{x - x^*}^2 + \sigma(x)^2}
+ \eta\beta\sqrt{2d}
\kh{\frac{\norm{x - x^*}}{\sqrt{\pi}\sigma(x)} + 1}
\kh{
\norm{x - x^*}^2 + 2\sigma(x)^2
}.
\end{aligned}
\end{equation}
Note that $f_{\lb}\le f^*$, so that by the definition of $\sigma(x)$ one has
$$
\sqrt{\eta s\kh{f(x) - f^*}} \le \sqrt{\eta s\kh{f(x) - f_{\lb}}}=\sigma(x) =
\sqrt{\eta s\kh{f(x) - f^*}+\eta s(f^*-f_{\lb})}.
$$
According to \eqref{eq:lm.nearly_convex} of \Cref{lm:nearly_convex}, we conclude that
\[
\label{eq:pf_key:sigma}
\sqrt{\frac{\eta s(\alpha - \beta)}{2}}\norm{x - x^*}
\le \sigma(x) \le
\sqrt{\frac{\eta sL}{2}\norm{x - x^*}^2 + \eta s(f^* - f_{\lb})}.
\]
Take \eqref{eq:pf_key:sigma} into \eqref{eq:expg.1}, one has
\begin{equation}
\label{eq:expg.2}
\begin{aligned}
&\expect{\xi}{\norm{G(x + \sigma(x)\xi) - x^*}^2} 
\le (1 - \eta\lambda)\kh{\kh{1 + \frac{\eta sL}{2}}\norm{x - x^*}^2 + \eta s(f^* - f_{\lb})}\\
&\quad+ \eta\beta\sqrt{2d}
\kh{\sqrt{\frac{2}{\pi\eta s(\alpha - \beta)}} + 1}
\kh{
(1 + \eta sL)\norm{x - x^*}^2 + 2\eta s(f^* - f_{\lb})
}.
\end{aligned}
\end{equation}
Denote $C_1 := \beta\sqrt{2d}\kh{\sqrt{\frac{2}{\pi\eta s(\alpha - \beta)}} + 1}$, and then we organize \eqref{eq:expg.2} as follows
\begin{equation}
\label{eq:expg.final}
\begin{aligned}
&\expect{\xi}{\norm{G(x + \sigma(x)\xi) - x^*}^2} 
\\
&\le \left[1 - \eta\kh{\lambda\kh{1 + \frac{1}{2}\eta sL} - C_1\kh{1 + \eta sL} - \frac{1}{2}sL} \right]\norm{x - x^*}^2\\
&\quad+ (1 - \eta(\lambda - 2C_1))\eta s(f^* - f_{\lb}).
\end{aligned}
\end{equation}
By \eqref{eq:asseq} we have $\beta = \beta(f) \le \frac{1}{4}\sqrt{\frac{\alpha^5}{dL^3}}$, $L \ge \alpha$ and $d \ge 1$.
One has
$ \beta \le \frac{1}{4\sqrt{d}}\sqrt{\frac{\alpha^5}{L^3}} \le \frac{\alpha}{4}$,
so that $\frac{L}{\alpha - \beta} \le \frac{4L}{3\alpha}$.
Meanwhile, by $\eta = \frac{2\alpha}{5L^2}$ and $\lambda = 2\alpha - \eta L^2 = \frac{8\alpha}{5}$ one has
$\eta\lambda = \frac{16\alpha^2}{25L^2}$.
Note that
$$
\sqrt{\frac{2}{\pi\eta s(\alpha - \beta)}} = \sqrt{\frac{2}{\pi} \frac{3L}{\alpha - \beta}
\frac{1}{\eta\lambda}}
\le 5\sqrt{\frac{L^3}{2\pi\alpha^3}},
$$
and it is easy to verify that $\frac{5}{\sqrt{\pi}} + \sqrt{2}<\frac{152}{35}$. Therefore,
\begin{equation}\label{eq:lm.tch}
C_1
\le \frac{\alpha}{4}\sqrt{\frac{\alpha^3}{L^3}}
\kh{5\sqrt{\frac{L^3}{\pi\alpha^3}} + \sqrt{2}}
\le \frac{\alpha}{4}\kh{\frac{5}{\sqrt{\pi}} + \sqrt{2}}
<\frac{38}{35}\alpha = \frac{19}{28}\lambda.
\end{equation}
Moreover, one has $\eta\lambda\le \frac{16}{25}$,
which, together with $s = \frac{\lambda}{3L}$, implies
$1 + \eta sL = \frac{\eta\lambda}{3} + 1 \le \frac{91}{75}$.
Consequently,
$$
C_1\kh{1 + \eta sL}
< \frac{19}{28} \cdot \frac{91}{75}\lambda
= \kh{\frac{99}{100} - \frac{1}{6}}\lambda.
$$
Taking this and \eqref{eq:lm.tch}, together with $sL = \frac{\lambda}{3}$, we get
\begin{equation}
\label{eq:coef.1}
\lambda\kh{1 + \frac{1}{2}\eta sL} - C_1\kh{1 + \eta sL} - \frac{1}{2}sL
\ge \lambda\kh{1 + \frac{\eta\lambda}{6}} - \kh{\frac{99}{100} - \frac{1}{6}}\lambda - \frac{\lambda}{6}
> \frac{\lambda}{100},
\end{equation}
\begin{equation}
\label{eq:coef.2}
(1 - \eta(\lambda - 2C_1))\eta s
\le \kh{1 - \eta\kh{\lambda - 2\frac{19}{28}\lambda}}\frac{\eta\lambda}{3L}
= \frac{\eta\lambda(5\eta\lambda + 14)}{42L}.
\end{equation}
By taking \eqref{eq:coef.1} and \eqref{eq:coef.2} into \eqref{eq:expg.final}, one can get
$$
\expect{\xi}{\norm{G(x + \sigma(x)\xi) - x^*}^2}
\le \kh{1 - \frac{\eta\lambda}{100}}\norm{x - x^*}^2
+ \frac{\eta \lambda(5\eta\lambda + 14)}{42 L} (f^* - f_{\lb}),
$$
which is \eqref{result:expect_decent}, and this completes the proof.
\end{proof}

We are ready to give the proof of \Cref{thm:main}.
\begin{proof}[\bf Proof of \Cref{thm:main}]
In accordance with \Cref{lm:key}, $G(x) = x - \eta\nabla f(x)$ is the gradient descent map.
Denote $ \omega_t:= \SG(x_t)-\nabla f(x_t)$ as the random noise so that by {\Cref{asp:sg}} we know that
\[
\label{eq:expomg}
\ex{\omega_t} = 0\quad \mbox{and}\quad \ex{\norm{\omega_t}^2} \le r^2.
\]
By letting $y_t := x_t - \eta\nabla f(x_t)$ and $\sigma(x):= \sqrt{\eta s\kh{f(x) - f_{\lb}}}$, one can rewrite the iterations in \Cref{alg:gngd} as follows
$$
 \left\{\begin{array}{l}
y_t = x_t - \eta\nabla f(x_t) = G(x_t),\\
x_{t + \frac{1}{2}} = x_t - \eta\SG(x_t) = y_t - \eta\omega_t,\\
x_{t + 1} = x_{t + \frac{1}{2}} - \sigma_t\xi_t
= x_{t + \frac{1}{2}} - \sigma(x_{t + \frac{1}{2}})\xi_t,
 \end{array}\right.
 \quad 0\le t\le T,
$$
which can be reorganized as
 \begin{equation}
 \label{eq:iteration}
 \left\{\begin{array}{l}
x_{t + \frac{1}{2}} = y_t - \eta\omega_t,\\
x_{t + 1} = x_{t + \frac{1}{2}} - \sigma(x_{t + \frac{1}{2}})\xi_t,\\
y_{t + 1} = G(x_{t + 1}),
 \end{array}\right.
 \quad \quad 0\le t\le T.
 \end{equation}
Note that the $\sigma$-fields $\mathcal{F}_t = \sigma\{\xi_0, \cdots, \xi_{t - 1}, \omega_0, \cdots, \omega_{t - 1}\}$
and $\mathcal{F}_{t+\frac{1}{2}} = \sigma\{\xi_0, \cdots, \xi_{t - 1}, \omega_0, \cdots, \omega_{t - 1},\omega_t\}$
constitute a generated filtration.
We know from \Cref{lm:nearly_convex} that $x^*$ is the unique global minima of $f$, so that by \eqref{eq:iteration} we have
$$
 \begin{aligned}
\ex{\norm{y_{t + 1} - x^*}^2 \mid \field{F}_t}
&= \ex{\ex{\norm{y_{t + 1} - x^*}^2 \mid \field{F}_{t+\frac{1}{2}} }\mid \field{F}_t}\\
&= \ex{\expect{\xi_t}{\norm{G\kh{x_{t + \frac{1}{2}} - \sigma(x_{t + \frac{1}{2}})\xi_t} - x^*}^2}\mid \field{F}_t}.
 \end{aligned}
$$
Then by using \Cref{lm:key}, \eqref{def:b} and \eqref{eq:expomg} we can get
$$
\begin{aligned}
\ex{\norm{y_{t + 1} - x^*}^2 \mid \field{F}_t}
&\le \ex{\kh{1 - \frac{\eta\lambda}{100}}\norm{x_{t + \frac{1}{2}} - x^*}^2 +
\frac{\eta\lambda(5\eta\lambda + 14)}{42L}(f^* - f_{\lb}) \mid \field{F}_t}\\
&= \kh{1 - \frac{\eta\lambda}{100}}\ex{\norm{y_t - \eta\omega_t - x^*}^2 \mid \field{F}_t} + \frac{\eta\lambda(5\eta\lambda + 14)}{42L}(f^* - f_{\lb})\\
&\le \kh{1 - \frac{\eta\lambda}{100}}\norm{y_t - x^*}^2 + \eta^2r^2
+ \frac{\eta\lambda(5\eta\lambda + 14)}{42L}(f^* - f_{\lb})\\
&= \kh{1 - \frac{\eta\lambda}{100}}\norm{y_t - x^*}^2 + \eta\lambda b.
\end{aligned}
$$
Consequently,
\[
\label{eq:main.final}
\ex{\norm{y_{t + 1} - x^*}^2 -100b \mid \field{F}_t}
\le
 \kh{1 - \frac{\eta\lambda}{100}}\kh{\norm{y_t - x^*}^2 -100b}.
\]
Taking expectation on both sides of \eqref{eq:main.final} and applying the inequality recursively for $t = 0, 1, \dots, t - 1$, we obtain
\[
\label{eq:main.finalplus}
\ex{\norm{y_t - x^*}^2}
\le
\kh{1 - \frac{\eta\lambda}{100}}^t\norm{y_0 - x^*}^2 + 100b,\quad \forall\,0 \le t \le T.
\]
In case $b = 0$, the inequality \eqref{eq:main.finalplus} leads to $\ex{\norm{y_T - x^*}^2} \le \kh{1 - \frac{\eta\lambda}{100}}^T\norm{y_0 - x^*}^2$.
Thus, for any $\varepsilon > 0$, Markov's inequality implies
$$
\proba{\norm{y_T - x^*}^2 > \varepsilon} \le \frac{1}{\varepsilon}\kh{1 - \frac{\eta\lambda}{100}}^T\norm{y_0 - x^*}^2.
$$
Consequently, for any $\zeta \in (0, 1)$, one has
$\norm{y_T - x^*}^2 \le \varepsilon$ with the probability at least being $1 - \zeta$, as long as
$$T \ge \frac{100}{\eta\lambda}\ln \frac{\norm{y_0 - x^*}^2}{\zeta\varepsilon}.$$

On the other hand, if $b > 0$, by $\eta\lambda\le \frac{16}{25}$,
we can use \eqref{eq:main.final} by considering the stochastic process $\{\norm{y_t - x^*}^2 - 100b\}_{t\in\mathds N}$.
More specifically, by applying \Cref{lm:st} with $\theta = 1 - \frac{\eta\lambda}{100}$ and $\ell = 100b$ to this stochastic process one can get that for any given nonnegative integer $T$,
$$
\proba{\exists t \le T, \norm{y_t - x^*}^2 < 200b}\ge 1 - \kh{\frac{200 - {\eta\lambda } }{200 }}^T\frac{B}{100b}
$$
with
$$
B = \kh{\norm{y_0 - x^*}^2 - 100b}\proba{ \norm{y_0 - x^*}^2 \ge 200b}
\le \norm{y_0 - x^*}^2.
$$
By taking $T \ge \frac{200}{\eta\lambda}\ln \frac{\norm{y_0 - x^*}^2}{100b\zeta}$ one obtains that 
$$
 1 - \kh{1 - \frac{\eta\lambda}{200}}^T\frac{B}{100b}
 \ge 1 - \kh{1 - \frac{\eta\lambda}{200}}^T\frac{\norm{y_0 - x^*}^2}{100b} \ge 1 - \zeta,
$$
 
Furthermore, note that $y_t = x_t - \eta\nabla f(x_t)$ and $\eta L \le 2/5$. One has from \eqref{eq:asp.L} that
\[
\label{eq:coro.0}
\norm{y_t - x^*} \ge \norm{x_t - x^*} - \eta \norm{\nabla f(x_t)}
\ge \norm{x_t - x^*} - \eta L\norm{x_t - x^*}
\ge (1 - \eta L)\norm{x_t - x^*}.
\]
Then, from \eqref{eq:lm.nearly_convex} of \Cref{lm:nearly_convex} we have
\begin{equation}\label{eq:coro.1}
f(x_t) - f^*
\le \frac{L}{2}\norm{x_t - x^*}^2
\le \frac{L}{2(1 - \eta L)^2}\norm{y_t - x^*}^2.
\end{equation}
According to \Cref{thm:main}, with probability at least $1 - \zeta$,
there exists $t \le T$ such that $\norm{y_t - x^*}^2 \le 200b$.
Thus, by \eqref{eq:coro.1} we have \eqref{eq:coro.2} holds, and it, together with \Cref{lm:nearly_convex}, implies
\eqref{result:coro.11}.
Similar to \eqref{eq:coro.0}, we can get
\[
\label{eq:coro.00}
\norm{y_t - x^*}
\le \norm{x_t - x^*} + \eta \norm{\nabla f(x_t)}
\le \norm{x_t - x^*} +\eta L\norm{x_t - x^*}
 \le (1 + \eta L)\norm{x_t - x^*}
 \].
Then by \eqref{result:coro.11} we know that
$$
\norm{y_{t_*} - x^*}^2
\le (1 + \eta L)^2 \frac{200Lb}{(1 - \eta L)^2(\alpha - \beta)},
$$
which implies \eqref{result:coro.2}, and this completes the proof.
\end{proof}

\subsection{Proof of Theorem \ref{thm:double_loop}}\label{subsec:proof.thm2}

\begin{proof}[Proof of \Cref{thm:double_loop}]
 Define
$$
 b^\nu = \frac{\eta r^2}{\lambda} + \frac{5\eta\lambda + 14}{42L}(f^* - f_{\lb}^\nu),\quad \forall \nu>0.
$$
Meanwhile, define the events
$$
 \mathfrak{E}_\nu = \{f^* - f_{\lb}^i > \varepsilon,\forall i<\nu\}\cap\{\forall i < \nu, \exists t_i \le T_2, \norm{y_{t_i}^i - x^*}^2 \le 200b^i\}, \quad\forall\nu \ge 0,
$$
so that $\mathfrak{E}_0$ is defined as the total probability space.
Obviously, $\mathfrak{E}_{\nu} \supset \mathfrak{E}_{\nu + 1}$.
Moreover, if $ \mathfrak{E}_\nu$ holds for a certain $\nu>0$, by \eqref{eq:coro.1} we know that
 \begin{equation}
 \label{eq:festimate}
 f^* - f_{\lb}^{\nu - 1} > \varepsilon \quad\mbox{and}\quad
 f(x_{\min}^{\nu - 1}) - f^*
 \le  \frac{L}{2(1 - \eta L)^2} \min_{t}  \norm{y^{\nu-1}_t - x^*}^2
 \le \frac{100Lb^{\nu - 1}}{(1 - \eta L)^2}.
 \end{equation}
Consequently, if $ \mathfrak{E}_\nu$ holds,  by using the facts that $b^\varepsilon = \frac{\eta r^2}{\lambda} + \frac{5\eta\lambda + 14}{42L}  \varepsilon$ and
$\gamma = \frac{(1 - \eta L)^2\varepsilon}{(1 - \eta L)^2\varepsilon + 100Lb^\varepsilon}$, one has
\[
\label{eq:difge}
\begin{aligned}
 f^* - f_{\lb}^\nu &= (1 - \gamma)\kh{f^* - f_{\lb}^{\nu - 1}} - \gamma \kh{f(x_{\min}^{\nu - 1}) - f^*}\\
 &\ge (1 - \gamma)\kh{f^* - f_{\lb}^{\nu - 1}} - \gamma \frac{100L}{(1 - \eta L)^2}
 b^{\nu - 1}
 \\
 &
 = \frac{ 100L }{(1 - \eta L)^2\varepsilon + 100L}
 \kh{\kh{f^* - f_{\lb}^{\nu - 1}} b^\varepsilon
 -   \varepsilon   b^{\nu - 1}}
  \\
 &
 = \frac{ 100L }{(1 - \eta L)^2\varepsilon + 100L}\frac{\eta r^2}{\lambda}
 \kh{ {f^* - f_{\lb}^{\nu - 1}}
 -  \varepsilon     } > 0.
     \end{aligned}
\]
Based on \eqref{eq:difge}, we define the events
$$
 \mathfrak{A}_\nu = \left\{0 \le f^* - f_{\lb}^\nu \le \varepsilon\right\}
 \quad\mbox{and}
 \quad
 \mathfrak{B}_\nu = \left\{ \norm{y_t^\nu - x^*}^2 > 200b^\nu, \forall t \le T_2\right\},
 \quad\forall \nu\ge 0.
$$
According to the above discussions we know that $\mathfrak{E}_{\nu} \backslash \mathfrak{E}_{\nu + 1} = \mathfrak{E}_{\nu} \cap (\mathfrak{A}_\nu \cup \mathfrak{B}_\nu)$. Note that $b^\nu > b^\varepsilon$ is equivalent to $f^* - f_{\lb}^\nu > \varepsilon$. Consequently, by defining
$$
 \mathfrak{B}_\nu' =
\begin{cases}
 \mathfrak{B}_\nu, & \nu=0,
 \\
  \left\{\forall t \le T_2, \norm{y_t^\nu - x^*}^2 > 200\cdot\max\{b^\nu, b^\varepsilon\}\right\},
& \forall\nu>0,
\end{cases}
$$
 one has  $\mathfrak{A}_\nu \cup \mathfrak{B}_\nu = \mathfrak{A}_\nu \cup \mathfrak{B}_\nu', \forall\nu>0$.
 Therefore,
 \begin{equation}\label{eq:thm2.proba}
 \begin{aligned}
\proba{\mathfrak{E}_{\nu}} &= \proba{\mathfrak{E}_{\nu + 1}} + \proba{\mathfrak{E}_{\nu} \cap (\mathfrak{A}_\nu \cup \mathfrak{B}_\nu')}\\
&= \proba{\mathfrak{E}_{\nu + 1}}
+ \proba{\mathfrak{E}_{\nu} \cap \mathfrak{A}_\nu \cap \mathfrak{B}_\nu'^c}
+ \proba{\mathfrak{E}_{\nu} \cap \mathfrak{B}_\nu'}\quad\forall \nu\ge0.
 \end{aligned}
 \end{equation}
In the following, we prove
\[
\label{eq:probkey}
\proba{\mathfrak{B}_\nu' \mid \mathfrak{E}_{\nu}} \le \zeta',\quad\forall \nu\ge 0.
\]
According to \Cref{thm:main} and the bound of $T_1$ given in \eqref{eq:t1t2}, we know that \eqref{eq:probkey} holds for $\nu=0$.
For the case that $\nu>0$,  we know that $\mathfrak{E}_{\nu}$ happens implies that \eqref{eq:difge} holds.
Moreover, it holds that
$$
 \proba{\mathfrak{B}_\nu' \mid \mathfrak{E}_{\nu}} \le  \proba{ \left\{\forall t \le T_2, \norm{y_t^\nu - x^*}^2 > 200\cdot  b^\nu \right\} \mid \mathfrak{E}_{\nu}} .
$$
Thus, if $b^\nu \ge  b^\varepsilon$,  by \Cref{thm:main} we know that \eqref{eq:probkey} holds if
\begin{equation}\label{eq:thm2.t2.1}
T_2
\ge \frac{200}{\eta\lambda}\ln \frac{\norm{y_0^\nu - x^*}^2}{100b^\varepsilon\zeta'}
\ge \frac{200}{\eta\lambda}\ln \frac{\norm{y_0^\nu - x^*}^2}{100b^\nu\zeta'}.
\end{equation}
On the other hand,  suppose that  $\nu>0$ and $b^\nu<b^\varepsilon$.
Similar to the proof of \Cref{thm:main}, we define
\begin{equation}
\label{eq:defg}
G^\nu_t = \norm{y_t^\nu - x^*}^2 - 100b^\nu \ge -100b^\nu.
\end{equation}
Therefore, with $\mathfrak{E}_{\nu}$ holds,  the generated filtration is given by
$$
\mathcal{F}^\nu_t:=
\sigma\{\xi^\nu_0, \cdots, \xi^\nu_{t - 1}, \omega^\nu_0, \cdots, \omega^\nu_{t - 1}\}, \quad\forall\nu>0, 0\le t\le T_2;
$$
We know that (cf. the derivation of \cref{eq:main.final}) that,
    \begin{equation*}
        \ex{G^\nu_{t + 1} \mid \mathcal{F}_t} \le \kh{1 - \frac{\eta\lambda}{100}} G^\nu_t,\quad \forall 0\le t\le T_2-1, \
         \forall\nu>0 \mbox{ with }
      \mathfrak{E}_{\nu} \mbox{ holds}.
    \end{equation*}
Taking $\theta = 1 - \frac{\eta\lambda}{100}$, $b=100b^\nu$ and $\ell=100b^\epsilon$ in \Cref{lm:st}, one can get
\begin{equation}
\label{eq:pgtemp}
\proba{\exists t \le T_2,G^\nu_t < 100b^\varepsilon}\ge 1 - \kh{ 1- \frac{\eta\lambda}{100} \frac{     b^\varepsilon}{ b^\nu + b^\varepsilon}}^T\frac{B}{100b^\varepsilon}
\quad\mbox{with}\quad
B: = \ex{G^\nu_0 \mathds{1}\{G^\nu_0 \ge 100b^\epsilon \}\mid   \mathfrak{E}_{\nu} }  \ge 0.
\end{equation}
From the \eqref{eq:defg} we know that $ B \le \norm{y_0^\nu - x^*}^2 $.
Meanwhile, $b^\nu < b^\varepsilon$ implies $\frac{b^\varepsilon}{b^\nu + b^\varepsilon} \ge \frac{1}{2}$.
Hence, \eqref{eq:pgtemp} implies
$$
\proba{\exists t \le T_2,G^\nu_t < 100b^\varepsilon}\ge 1 - \kh{1-  \frac{\eta\lambda}{200}}^T\frac{\norm{y_0^\nu - x^*}^2 }{100b^\varepsilon}.
$$
Therefore, \eqref{eq:probkey} holds if one has
\[
\label{eq:thm2.t2.2}
T_2 \ge \frac{200}{\eta\lambda}\ln \frac{\norm{y_0^\nu - x^*}^2}{100b^\varepsilon\zeta'},
\]
which, together with \eqref{eq:thm2.t2.1}, implies that \eqref{eq:thm2.t2.2} is sufficient for \eqref{eq:probkey}.

  Note that whenever $\mathfrak{E}_{\nu}$ holds,   $x_0^\nu = x_{\min}^{\nu - 1}$. Then, by using \eqref{eq:coro.00} and    \Cref{lm:nearly_convex}  one can  obtain
 \begin{equation*}
 \norm{y_0^\nu - x^*}^2
 \le(1+\eta L)^2\|x_0^\nu-x^*\|^2
  \le\frac{2(1+\eta L)^2}{\alpha-\beta}
  \kh{f(x_{\min}^{\nu - 1})-f^*}
 \end{equation*}
 Taking \eqref{eq:festimate} into account, one gets
$$
 \norm{y_0^\nu - x^*}^2 \le \kh{\frac{1 + \eta L}{1 - \eta L}}^2\frac{200L b^{\nu-1}}{\alpha - \beta}.
$$
Moreover,  since $f(x_{\min}^{\nu - 1}) \ge f^*$ and $f_{\lb}^{\nu } = (1 - \gamma)f_{\lb}^{\nu-1} + \gamma f\kh{x_{\min}^{\nu-1}}$, one has
 \begin{equation}\label{eq:thm2.flb}
 f^* - f_{\lb}^\nu = (1 - \gamma)\kh{f^* - f_{\lb}^{\nu - 1}} - \gamma \kh{f(x_{\min}^{\nu - 1}) - f^*}
 \le f^* - f_{\lb}^{\nu - 1}.
 \end{equation}
Thus, by $\mathfrak{E}_{0}\supset\cdots\supset \mathfrak{E}_{\nu-1} \supset \mathfrak{E}_{\nu}$ one has
$b^{\nu - 1} \le b^0$, so that
 \begin{equation*}
 \frac{200}{\eta\lambda}\ln \frac{\norm{y_0^\nu - x^*}^2}{100b^\varepsilon\zeta'}
 \le \frac{200}{\eta\lambda}\ln \kh{\kh{\frac{1 + \eta L}{1 - \eta L}}^2 \frac{2Lb^0}{(\alpha - \beta)b^\varepsilon\zeta'}}
 \le T_2.
 \end{equation*}
Therefore, \eqref{eq:probkey} always holds.

Based on \eqref{eq:probkey}, by using \eqref{eq:thm2.proba} one can get
 \begin{equation}\label{eq:thm2.proba.2}
 \proba{\mathfrak{E}_\nu} - \proba{\mathfrak{E}_{\nu + 1}}
 =
 \proba{\mathfrak{E}_{\nu} \cap \mathfrak{A}_\nu \cap \mathfrak{B}_\nu'^c}
+ \proba{\mathfrak{E}_{\nu} \cap \mathfrak{B}_\nu'}
 \le \proba{\mathfrak{E}_\nu \cap \mathfrak{A}_\nu \cap \mathfrak{B}_\nu'^c} + \zeta'.
 \end{equation}
  Summing \eqref{eq:thm2.proba.2} up with $\nu$ from $0$ to $N$ and using the fact that $\proba{\mathfrak{E}_{0}}=1$, one has
$$
 \proba{\mathfrak{E}_{N + 1}} + \sum_{\nu = 0}^N \proba{\mathfrak{E}_{\nu} \cap \mathfrak{A}_\nu \cap \mathfrak{B}_\nu'^c} \ge 1 - (N + 1)\zeta'
 = 1 - \zeta
$$
 Note that $\mathfrak{E}_{\nu} \cap \mathfrak{A}_\nu \cap \mathfrak{B}_\nu'^c \subseteq \mathfrak{E}_{\nu} \backslash \mathfrak{E}_{\nu + 1}$ for all $\nu\ge 0$, so that the sets $\mathfrak{E}_{\nu} \cap \mathfrak{A}_\nu \cap \mathfrak{B}_\nu'^c$ are separated for all $\nu\ge 0$.
Furthermore, each event $\mathfrak{E}_{\nu} \cap \mathfrak{A}_\nu \cap \mathfrak{B}_\nu'^c \subseteq \mathfrak{E}_{\nu} \backslash \mathfrak{E}_{\nu + 1}$ means that there exists $t \le T_2$ (or $t\le T_1$ when $\nu=0$) such that $\norm{y_t^\nu - x^*}^2 \le 200b^\varepsilon$, which yields from assentation (b) of \Cref{thm:main} that $f\kh{x_{\min}^\nu} - f^* \le \frac{100Lb^\varepsilon}{(1 - \eta L)^2}$.
Due to the monotonicity fact that $f(x_{\min}^N)\le f(x_{\min}^\nu)$, we obtain $f\kh{x_{\min}^N} - f^* \le \frac{100Lb^\varepsilon}{(1 - \eta L)^2}$.
Therefore, the last statement that should be proved is
 \begin{equation}
 \label{eq:final}
 \mathfrak{E}_{N + 1} \Rightarrow f\kh{x_{\min}^N} - f^* \le \frac{100Lb^\varepsilon}{(1 - \eta L)^2},
 \end{equation}
so that by \Cref{lm:nearly_convex} one has $\norm{x_{\min}^N - x^*}^2 \le \frac{200Lb^\varepsilon}{(1 - \eta L)^2(\alpha - \beta)}$.

Suppose on the contrary that \eqref{eq:final} is not true.
Since $f(x_{\min}^N)\le f(x_{\min}^\nu)$ for all $\nu\ge 0$,
we know that  $f\kh{x_{\min}^\nu} - f^* \ge \frac{100Lb^\varepsilon}{(1 - \eta L)^2}$ holds for all $\nu=1,\ldots, N$.
Therefore, from  \eqref{eq:thm2.flb} we know that
$$
f^* - f_{\lb}^\nu = (1 - \gamma)\kh{f^* - f_{\lb}^{\nu - 1}} - \gamma \kh{f(x_{\min}^{\nu - 1}) - f^*}
\le (1 - \gamma)\kh{f^* - f_{\lb}^{\nu - 1}} - \gamma \frac{100Lb^\varepsilon}{(1 - \eta L)^2},
$$
 which is equivalent to
  \begin{equation*}
\kh{f^* - f_{\lb}^\nu + \frac{100Lb^\varepsilon}{(1 - \eta L)^2}} \le (1 - \gamma) \kh{f^* - f_{\lb}^{\nu - 1} + \frac{100Lb^\varepsilon}{(1 - \eta L)^2}}.
 \end{equation*}
 Therefore,  as $ \mathfrak{E}_{N + 1}$ implies $ f^* - f_{\lb}^{N} > \varepsilon$  via \eqref{eq:festimate},
 \begin{equation*}
 \varepsilon + \frac{100Lb^\varepsilon}{(1 - \eta L)^2} \le (1 - \gamma)^N \kh{f^* - f_{\lb}^0 + \frac{100Lb^\varepsilon}{(1 - \eta L)^2}}.
 \end{equation*}
Then by using
$ \gamma = \frac{(1 - \eta L)^2\varepsilon}{(1 - \eta L)^2\varepsilon + 100Lb^\varepsilon} = \frac{\varepsilon}{\varepsilon + \frac{100Lb^\varepsilon}{(1 - \eta L)^2}}$ given in \eqref{eq:defgamma}, we have
 \begin{equation*}
 \frac{\varepsilon}{\gamma} \le (1 - \gamma)^N \kh{f^* - f_{\lb}^0 - \varepsilon + \frac{\varepsilon}{\gamma}}
 \Leftrightarrow (1 - \gamma)^{-N} \le (1 - \gamma) + \gamma \cdot \frac{f^* - f_{\lb}^0}{\varepsilon},
 \end{equation*}
which contradicts to $N \ge \frac{1}{\gamma}\ln\kh{(1 - \gamma) + \gamma\cdot\frac{f^* - f_{\lb}^0}{\varepsilon}}$.
 Therefore, \eqref{eq:final} holds.

Overall, we have proved that, with probability at least $1 - \zeta$, one has $\norm{x_{\min}^N - x^*}^2 \le \frac{200Lb^\varepsilon}{(1 - \eta L)^2(\alpha - \beta)}$, and this completes the proof of the theorem.
 \end{proof}

\section{Numerical experiments}
\label{sec:numerical_experiments}
In this section, we conduct numerical experiments to test the performance of the proposed GND (\Cref{alg:gngd}) and DL-GND (\Cref{alg:dlgnd}) algorithms. 
By sampling 10,000 initial points, we run the algorithms and evaluate their performance using the mean squared error (MSE) and the non-convergence probability (N-CP). These two measures are defined by
\begin{equation}\label{eq:exp.metrics}
    \text{MSE}(t) = \mathbb{E}(\|x_t-x^*\|^2) \quad\mbox{and}\quad \text{N-CP}(t)= \mathbb{P}(\|x_t-x^*\|>10^{-3}),
\end{equation}
where $x_t$ is the random vector of $t$-th iteration and $x^*$ is the global minima. 
The MSE quantifies convergence in the state space, while the N-CP quantifies convergence in the probability space.

\subsection{Minimizing a univariate nearly convex function}
Recall that the objective function in \Cref{example:nearly_convex} is given by
\begin{equation}\label{obj:nearly_convex_repeat}
J_{n, k}^1(x) := \frac{1}{2}x^2 - \kh{1 + \frac{1}{k}}\int_{0}^{x} t(\sin t)^{2n} \dif t, \quad x\in\real,
\end{equation}
where $n,k\in\mathbb{N}^+$. Here, we consider two cases: $n=7, k=1$ and $n=112, k=2$. Notably, we can prove that $J_{112,2}^1$ is $\left(\frac{21}{25}, 1\right)$-nearly convex. However, it remains undetermined whether $J_{7,1}^1$ is nearly convex as the theoretical bound in \eqref{eq:asseq} may not be satisfied.

\begin{figure}
\centering
\begin{subfigure}{0.48\textwidth}
    \centering
    \includegraphics[width=.8\textwidth]{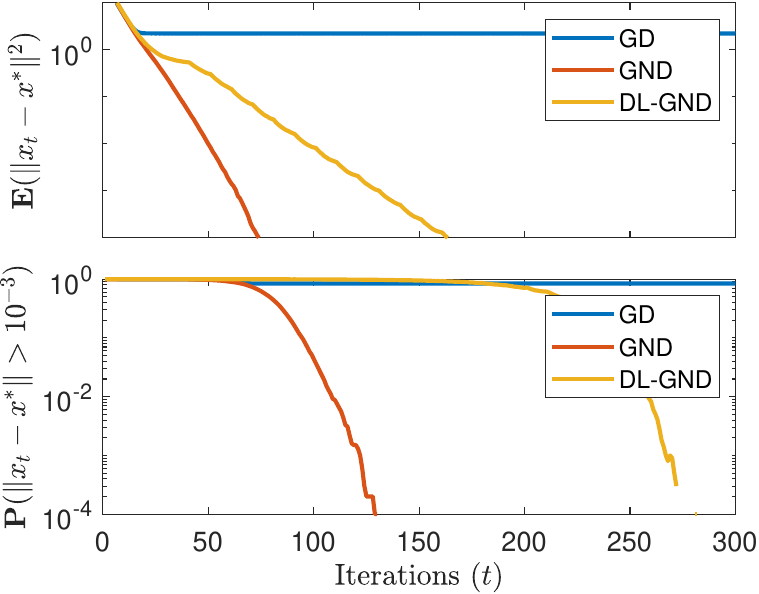}
    \caption{$n=112, k=2$}
\end{subfigure}
\begin{subfigure}{0.48\textwidth}
    \centering
    \includegraphics[width=.8\textwidth]{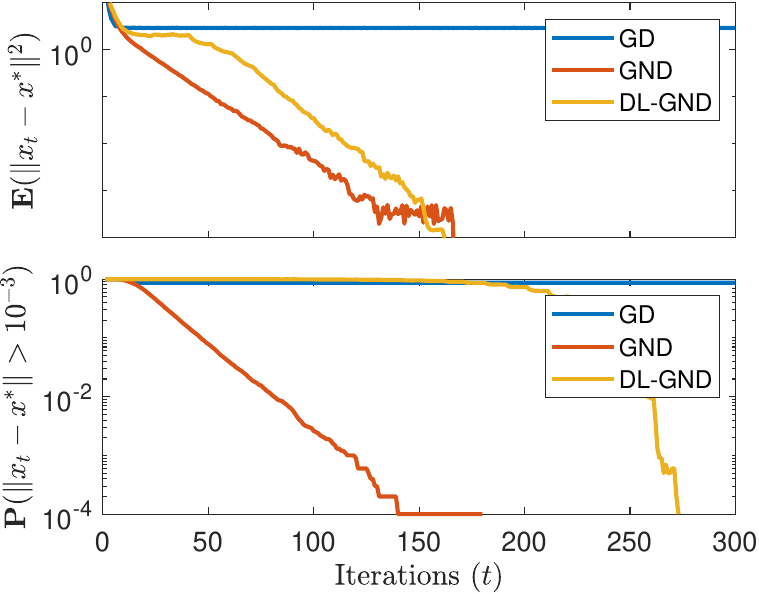}
    \caption{$n=7, k=1$}
\end{subfigure}
\caption{
\small
The MSE and N-CP versus iterations of the GD, GND, and DL-GND algorithms are evaluated by minimizing $J_{7, 1}^1(x)$ and $J_{112, 2}^1(x)$ given by \Cref{obj:nearly_convex_repeat}.
}
\label{fig:1d1}
\end{figure}

We run the GND and DL-GND algorithms, choosing 10,000 random initial points from the interval $[-10, 10]$. In this experiment, we test the full gradient case, that is, $\SG(x) = \nabla J_{n, k}^1(x)$ and $r = 0$. We set $\eta = 0.4, s = 0.5$ for the case $n=7, k=1$ and $\eta = 0.1, s = 0.2$ for the case $n=112, k=2$. The difference in parameter settings arises as the oscillation of $J_{7,1}^1(x)$ is more severe, requiring larger values for $\eta$ and $s$. In the GND algorithm, we set $f_{\lb} = 0$. For the DL-GND case, we set $f_{\lb} = -1$ and take $\gamma = 0.5, T_1 = 40, T_2 = 10, N = 30$.

The numerical behaviors of GND and DL-GND are given in \Cref{fig:1d1} in terms of the MSE and the N-CP. Both GND and DL-GND (\Cref{alg:gngd,alg:dlgnd}) algorithms are observed to converge linearly to the global optimal point in both state space and probability space, while most trajectories of GD fail to converge and get trapped in local minima. Due to the additional estimation of the lower bound of the objective function value, DL-GND requires more iterations than GND, but achieves similar accuracy.

These results properly represent our theoretical findings (\Cref{thm:main,thm:double_loop}). In addition, although we can not prove the nearly convexity of $J_{7, 1}^1(x)$, the numerical results show the effectiveness of GND and DL-GND algorithms, suggesting potential for future research on identifying weaker conditions for global convergence.

\subsection{Minimizing Rastrigin functions}
In this subsection, we test the performance of GND and DL-GND algorithms for minimizing the Rastrigin function \cite{Rastrigin1974systems}:
\begin{equation}\label{obj:Rastrigin}
J_{a, b, c, d}^3(\bm{x}) := a\kh{d-\sum_{i=1}^d \cos \left(b x^i\right)}+c \sum_{i=1}^d (x^i)^2,\quad \bm{x}=(x^1,\ldots, x^d)^T\in\real^d.
\end{equation}
We set $a = b = 1$ and vary the dimension $d$ and the coefficient $c$. For notational simplicity, we denote $J_{c, d}^3 := J_{1, 1, c, d}^3$.

\subsubsection{A two-dimensional case}
We perform numerical experiments to minimize the Rastrigin function $J_{c, d}^3(x)$ defined in \eqref{obj:Rastrigin} with $d=2$, focusing on the cases $c = 0.01$ and $c = 0.05$. The proposed GND and DL-GND algorithms are applied to $J_{c, d}^3(x)$, and their performance is evaluated using MSE and N-CP as defined in \eqref{eq:exp.metrics}. We adopt the following common setup for all experiments: Each consists of 10,000 independent and identically distributed trajectories to estimate the MSE and the N-CP. Each trajectory originates from a uniform distribution over the region $[-20, 20]^2$. Recall that $x_t$ is the random vector of $t$-th iteration. The experiments are divided into the following two parts.

\subsubsubsection{Part I: Default performance evaluation.}
In this part, we test the GND and DL-GND algorithms with their default optimal parameters, given in \Cref{tbl:para.default}. It is noted that in GND algorithm, $f_{\lb}=f^*=0$.

\begin{table}
\centering
\caption{Default parameter configurations for minimizing Rastrigin functions $J_{c, d}^3(x)$ in \eqref{obj:Rastrigin}.}
\label{tbl:para.default}
\begin{tabular}{|c|c|c|c|c|c|c|c|c|c|}
    \hline
    &&&&&&&&&\\[-10pt]
    Dimension                   & $c$                           & Algorithm &$\eta$ & $s$   & $f_{\lb}$ & $f_{\lb}^0$   & $\gamma$  & $T_1$ & $T_2$ \\
    \hline
    \multirow{4}{*}{$d = 2$}    & \multirow{2}{*}{$c = 0.05$}   & GND       & 1.5   & 2.0   & 0 & -             & -         & -     & -     \\\cline{3-10}
                                &                               & DL-GND    & 1.5   & 1.5   & - & -20           & 0.3       & 100   & 10    \\\cline{2-10}
                                & \multirow{2}{*}{$c = 0.01$}   & GND       & 1.5   & 4.0   & 0 & -             & -         & -     & -     \\\cline{3-10}
                                &                               & DL-GND    & 1.5   & 3.0   & - & -20           & 0.03      & 100   & 10    \\
    \hline
    \multirow{4}{*}{$d = 10$}   & \multirow{2}{*}{$c = 0.05$}   & GND       & 1.5   & 1.5   & 0 & -             & -         & -     & -     \\\cline{3-10}
                                &                               & DL-GND    & 1.5   & 1.4   & - & -20           & 0.025     & 100   & 10    \\\cline{2-10}
                                & \multirow{2}{*}{$c = 0.03$}   & GND       & 1.5   & 2.5   & 0 & -             & -         & -     & -     \\\cline{3-10}
                                &                               & DL-GND    & 1.5   & 1.5   & - & -20           & 0.0035    & 100   & 10    \\
    \hline
\end{tabular}
\end{table}

\begin{figure}
    \centering
    \begin{subfigure}{0.48\textwidth}
        \centering
        \includegraphics[width=.8\textwidth]{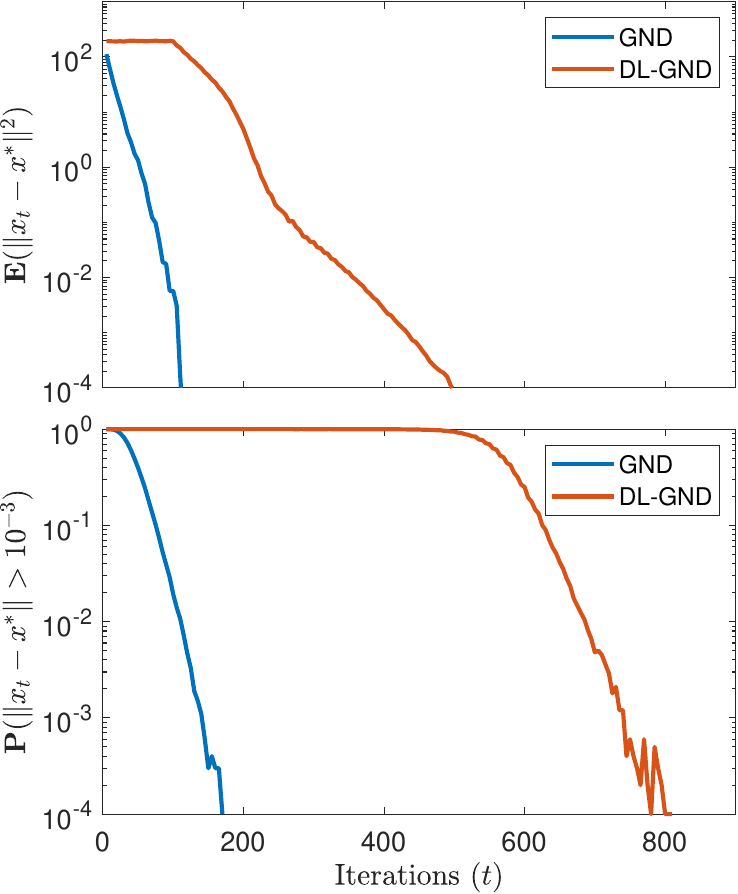}
        \caption{$c=0.05, d=2$}
    \end{subfigure}
    \begin{subfigure}{0.48\textwidth}
        \centering
        \includegraphics[width=.8\textwidth]{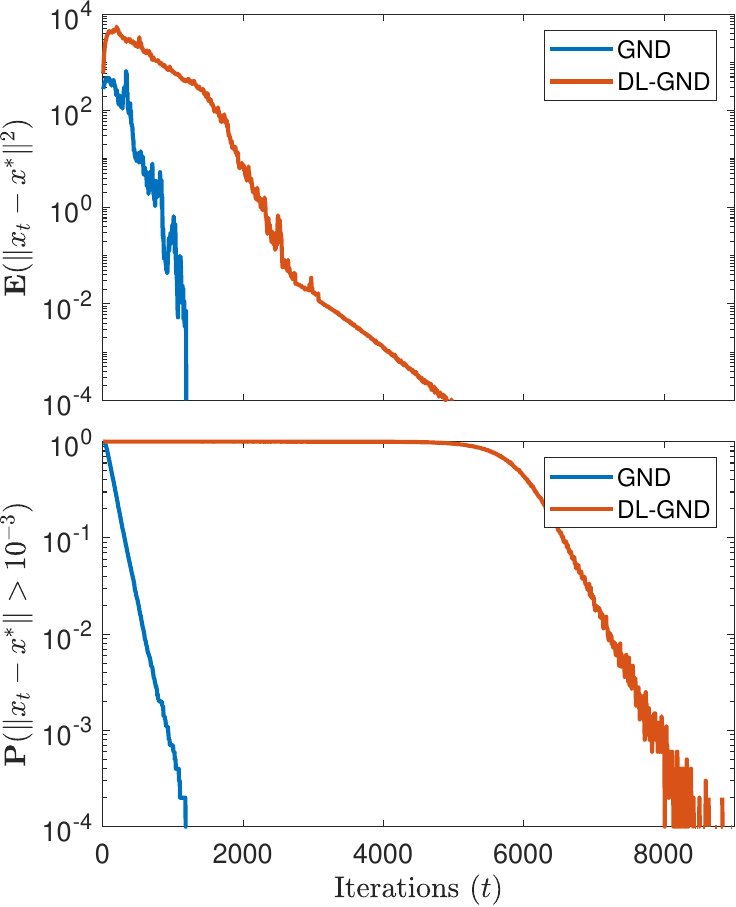}
        \caption{$c=0.01, d=2$}
    \end{subfigure}
    \caption{The convergence behavior for minimizing the two dimensional Rastrigin function $J_{c, d}^3(x)$ in \eqref{obj:Rastrigin} with $c = 0.05$ and $0.01$. The first and the second rows are the MSE and N-CP versus iterations, respectively.}
    \label{fig:2d.default}
\end{figure}

\begin{figure}
    \centering
    \includegraphics[width=.8\textwidth]{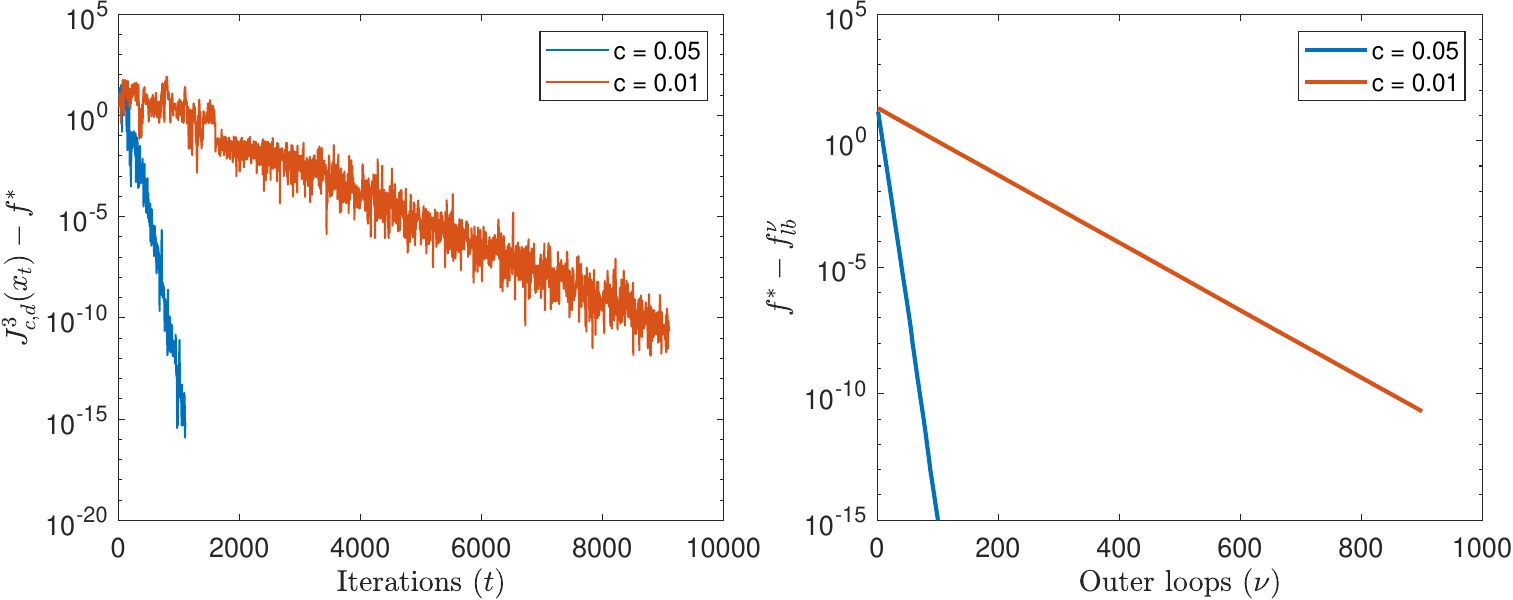}
    \caption{One trajectory for minimizing the two dimensional Rastrigin function $J_{c, d}^3(x)$ in \eqref{obj:Rastrigin} with $c = 0.01$ and $0.05$ using the DL-GND algorithms. The convergence behavior of the function value (Left) and the estimated lower bound of $f_{\lb}^\nu$ (Right).
    }
    \label{fig:2d4}
\end{figure}

The numerical results of GND and DL-GND, shown in \Cref{fig:2d.default}, are evaluated based on the metrics MSE and N-CP. The results show that both GND and DL-GND achieve linear convergence to the global minimum for $c = 0.05$ and $c = 0.01$, in both the state space and the probability space. However, the DL-GND algorithm converges 4–8 times more slowly than the GND algorithm for both objectives. This slower convergence is reasonable, given that the initial gap $f^* - f_{\lb}^0 = 20$ in DL-GND is significantly larger compared to the zero gap in GND.

\Cref{fig:2d4} illustrates a single trajectory of DL-GND, showing both the function values ($J_{c, d}^3(x_t)$) and the adaptive lower bounds ($f_{\lb}^\nu$). The result reveals that the lower bound estimate $f_{\lb}^\nu$ converges linearly to $f^*$, while the function value $J_{c, d}^3(x_t)$ also progressively converges to $f^*$, accompanied by noise-induced oscillations.

The experiments in the first part demonstrate that both GND and DL-GND efficiently achieve global convergence. Notably, for DL-GND, even in the absence of prior knowledge of the exact value of $f^*$, the adaptive strategy guarantees the linear convergence of $f_{\lb}^\nu$ to $f^*$.

\subsubsubsection{Part II: Parameter sensitivity analysis.}
In this part, we investigate the impact of various parameter choices on the convergence of the algorithms.
The GND and DL-GND algorithms are applied to optimize $J_{c, d}^3(x)$ in \eqref{obj:Rastrigin} with $d = 2$ and $c = 0.01$. 
The GND algorithm involves three hyperparameters: the step size or step size $\eta$, the noise coefficient $s$, and the lower bound of the function value $f_{\lb}$. We compare the results obtained by varying the hyperparameter $\eta, s$, and $f_{\lb}$, as shown in \Cref{fig:2d1,fig:2d2}. For the DL-GND algorithm, the key parameter influencing the convergence is $\gamma$, which controls the lower bound estimation of the function value. The corresponding results are presented in \Cref{fig:2d.gamma}.

\begin{figure}
    \centering
    \includegraphics[width=.8\textwidth]{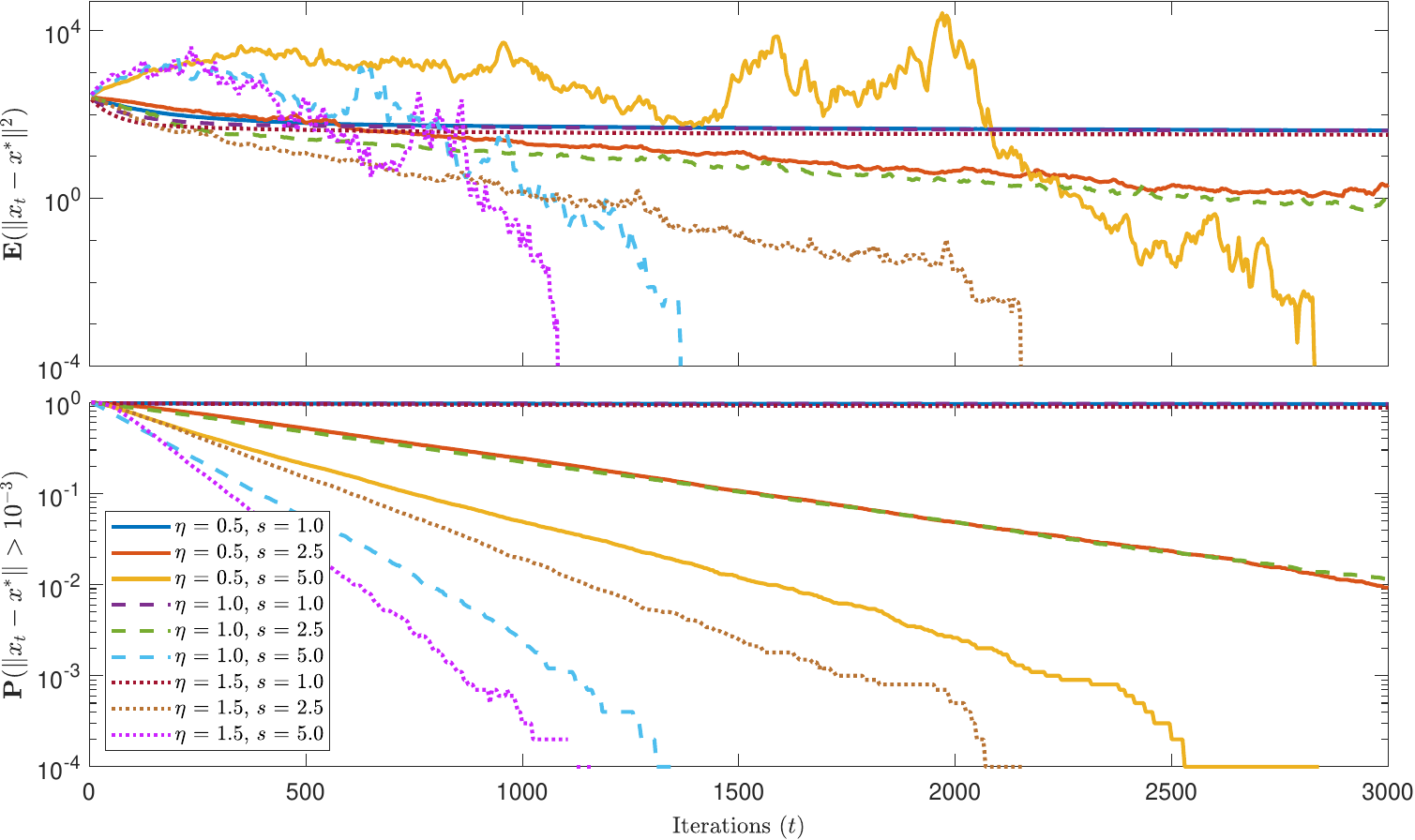}
    \caption{Minimizing the two dimensional Rastrigin function $J_{c, d}^3(x)$ in \eqref{obj:Rastrigin} with $c = 0.01$, using the GND algorithms. We fix $f_{\lb}=f^*=0$ and vary the two parameters $\eta$ and $s$.
    }
    \label{fig:2d1}
\end{figure}

\begin{figure}
    \centering
    \includegraphics[width=.8\textwidth]{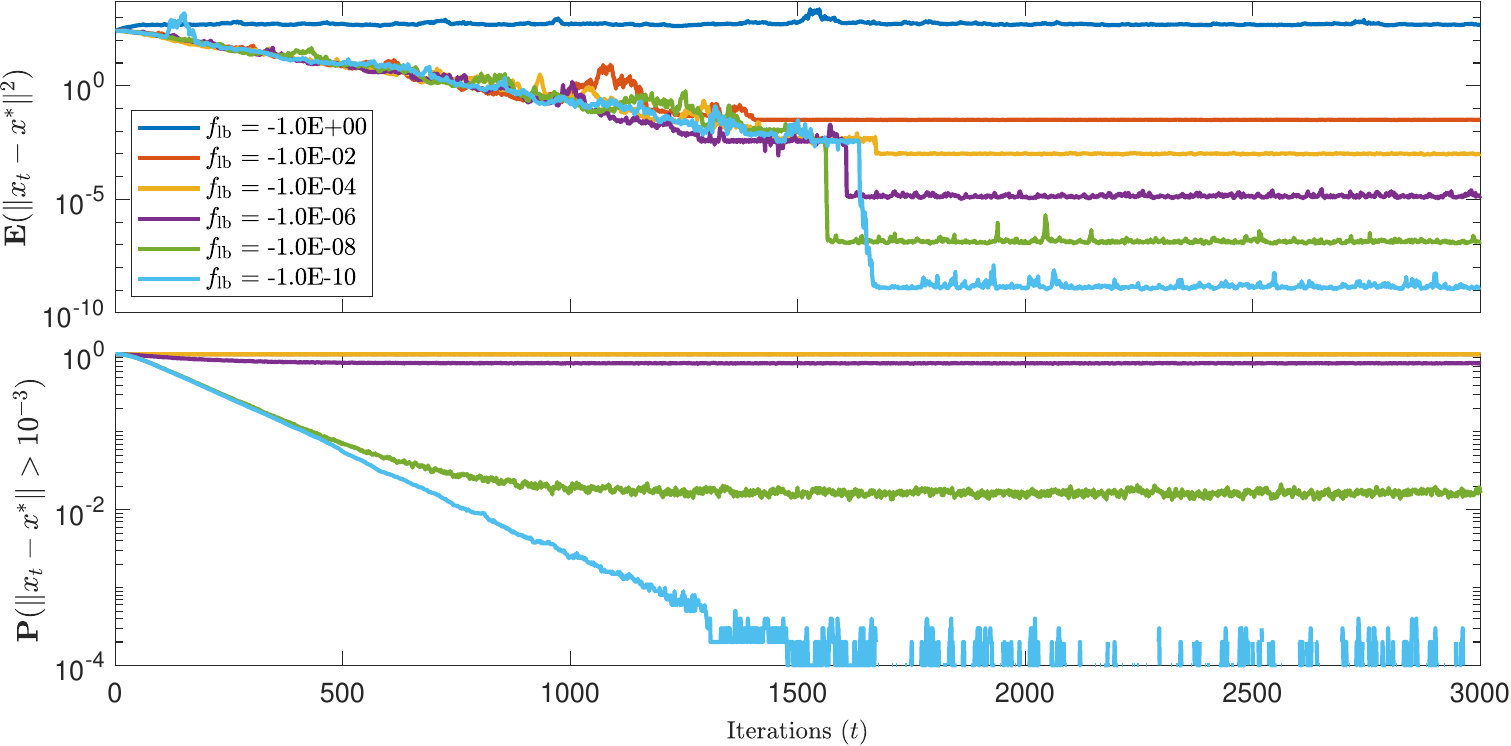}
    \caption{
        Minimizing the two-dimensional Rastrigin function $J_{c, d}^3(x)$ in \eqref{obj:Rastrigin} with $c = 0.01$, using the GND algorithms. We fix $\eta = 1.5, s = 3.0$ and vary $f_{\lb}$ from $-1$ to $-10^{-10}$.
    }
    \label{fig:2d2}
\end{figure}

\begin{figure}
    \centering
    \includegraphics[width=.8\textwidth]{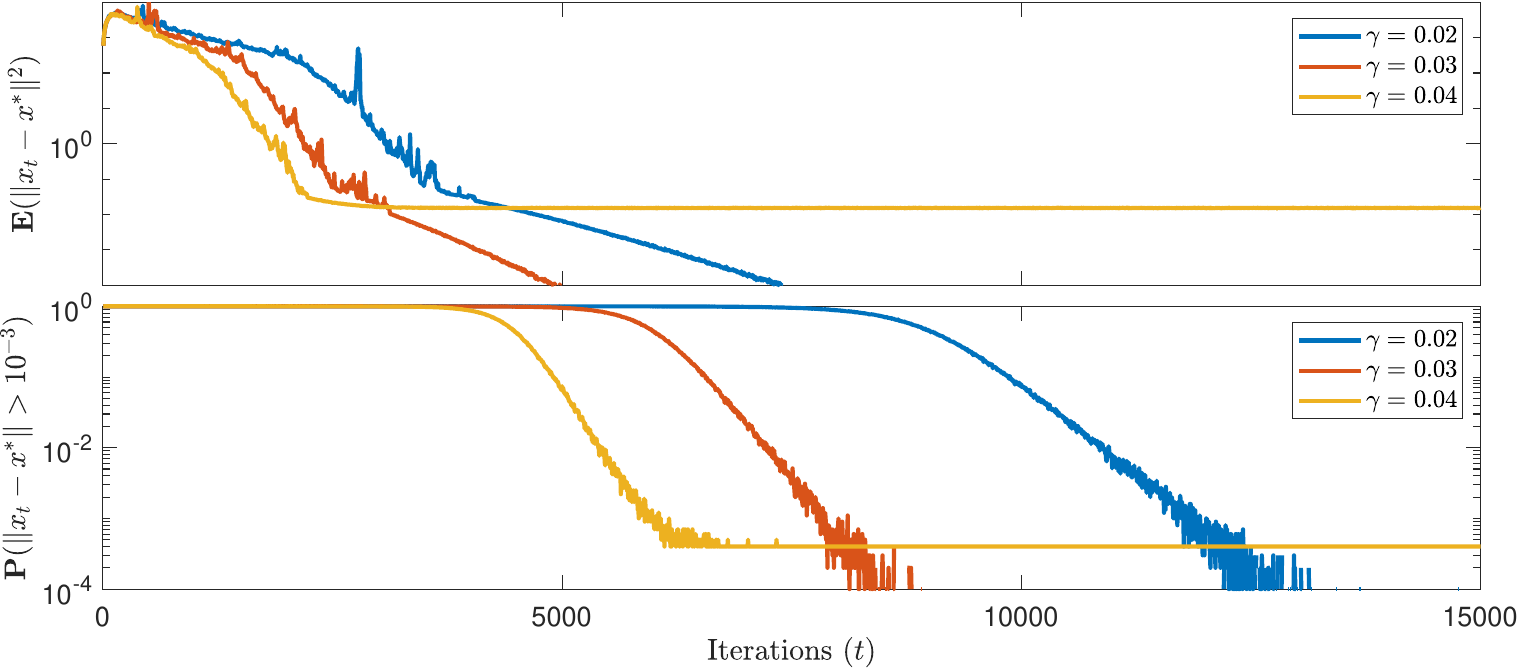}
    \caption{Minimizing the two-dimensional Rastrigin function $J_{c, d}^3(x)$ in \eqref{obj:Rastrigin} with $c = 0.01$, using the DL-GND algorithms. We fix $\eta = 1.5, s = 3.0, f_{\lb}^0 = -20, T_1 = 100, T_2 = 10$ and vary $\gamma$ from $0.02$ to $0.04$.}
    \label{fig:2d.gamma}
\end{figure}

In \Cref{fig:2d1}, the GND algorithm is applied with the lower bound fixed at $f_{\lb} = f^* = 0$, while varying the values of $\eta$ and $s$. The hyperparameters are selected from the set
\begin{equation*}
    (\eta, s) \in \{0.5, 1.0, 1.5\} \times \{1.0, 2.5, 5.0\}.
\end{equation*}
The estimated MSE and N-CP for each combination of $(\eta, s)$ are summarized in \Cref{fig:2d1}. 
The results indicate that GND converges linearly to the global minimum for all choices of $\eta$ when $s \ge 2.5$. However, when the noise coefficient $s$ is too small, the algorithm fails to converge to the global minimum, suggesting that insufficient noise makes it harder for the trajectory to escape local minima. When $s$ is chosen suitably large, the choice of $\eta$ and $s$ affects the convergence speed. Larger $\eta$ and $s$ appear to converge faster, but the converging trajectories become noisier. To achieve satisfactory performance, $\eta$ and $s$ should not be too large. 

In \Cref{fig:2d2}, the GND algorithm is applied with fixed $\eta = 1.5, s = 3.0$, 
while varying $f_{\lb}$ from $-1$ to $-10^{-10}$. The results indicate that $f_{\lb}$ plays a critical role in the global convergence of GND. Specifically, convergence to within 0.001 of $x^*$ is not achieved until $f^* - f_{\lb} \le 10^{-8}$.

In \Cref{fig:2d.gamma}, we apply the DL-GND algorithm with fixed parameters $\eta = 1.5, s = 3.0, f_{\lb}^0 = -20$ and varying $\gamma$ from $0.02$ to $0.04$. The number of iterations is set to $T_1 = 100$ for the initial stage and $T_2 = 10$ for the inner loop. The results indicate that $\gamma$ significantly affects the convergence of the DL-GND algorithm. When $\gamma$ is too small, the convergence speed is excessively slow. Conversely, when $\gamma$ is too large, some trajectories fail to converge as the lower bound estimations $f_{\lb}^\nu$ grow too quickly, occasionally surpassing $f^*$ in some outer loops. Therefore, selecting an appropriate $\gamma$ is essential to achieve efficient convergence in DL-GND.

The experiments in the second part demonstrate that, with appropriate choices of hyperparameters $\eta$ and $s$, and an accurate lower bound estimation of $J_{c, d}^3$, the GND algorithm can efficiently converge to a neighborhood of the global minimum for certain nonconvex functions. While the DL-GND algorithm eliminates the need for an accurate lower bound estimation, it requires a properly chosen $\gamma$ to efficiently update the lower bound estimation while ensuring the constraint $f_{\lb}^{\nu} \le f^*$ is satisfied. In practical implementations, these hyperparameters can be tuned to achieve optimal performance.

\subsubsection{A ten-dimensional case}
We perform experiments to minimize the 10-dimensional Rastrigin function $J_{c, d}^3(x)$ ($d = 10$) defined in \eqref{obj:Rastrigin} with $c = 0.05$ and $c = 0.03$. The proposed GND and DL-GND algorithms are applied to $J_{c, d}^3(x)$, and their performance is evaluated using mean squared error (MSE) and non-convergence probability (N-CP) as defined in \eqref{eq:exp.metrics}.

For the GND algorithm, we set $f_{\lb} = f^* = 0$, while for the DL-GND algorithm, we set $f_{\lb}^0 = -20$. All other hyperparameters, including $\eta, s, \gamma$, are treated as tunable parameters, with the detailed parameter configuration provided in \Cref{tbl:para.default}. The remaining experimental setup is analogous to the 2-dimensional case. The results, evaluated in terms of MSE and N-CP, are displayed in \Cref{fig:10d}.

From \Cref{fig:10d}, we observe that both the GND and DL-GND algorithms successfully converge to the global minimum in both settings. In the 10-dimensional space, the convergence speed is nearly 10 times slower compared to the 2-dimensional case. This slowdown is expected, as the number of local minima and saddle points grows exponentially with the dimension. These results highlight the robustness and effectiveness of the proposed GND and DL-GND algorithms in handling high-dimensional nonconvex optimization problems.

\begin{figure}
    \centering
    \begin{subfigure}{0.48\textwidth}
        \centering
        \includegraphics[width=\textwidth]{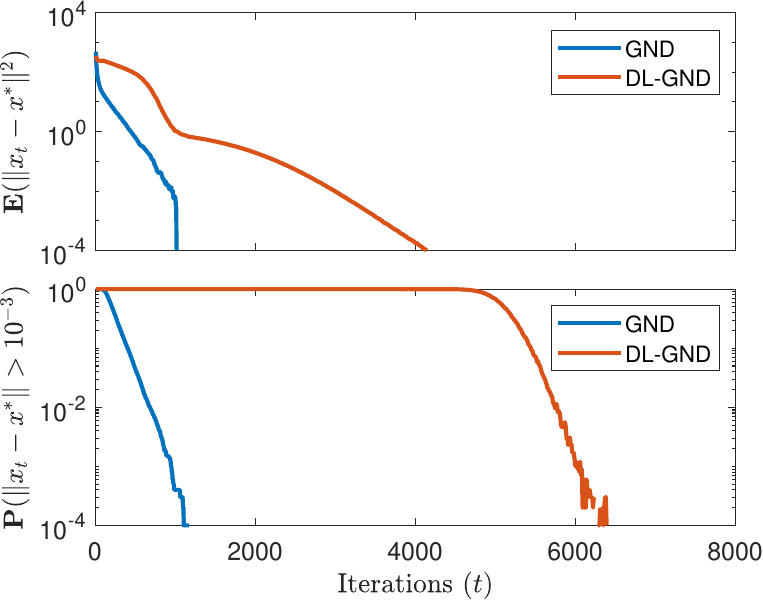}
        \caption{$c=0.05, d=10$}
    \end{subfigure}
    \begin{subfigure}{0.48\textwidth}
        \centering
        \includegraphics[width=\textwidth]{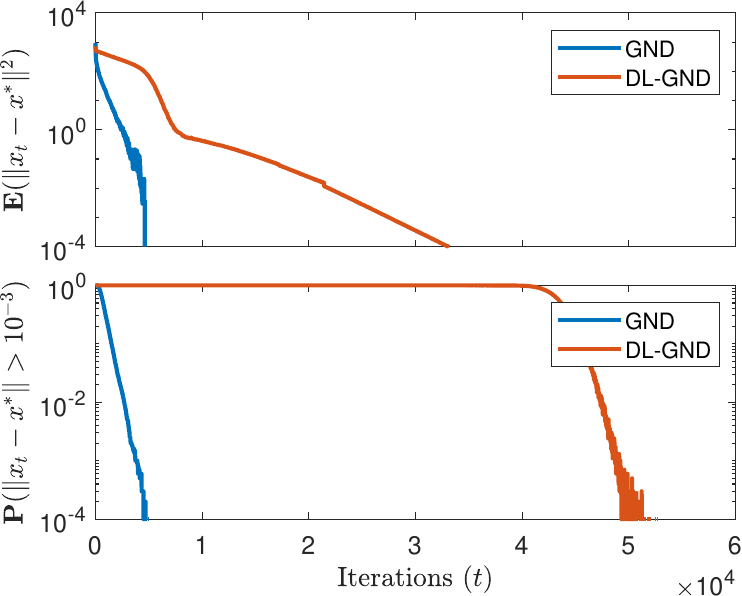}
        \caption{$c=0.03, d=10$}
    \end{subfigure}
    \caption{
        The convergence behavior for minimizing the ten-dimensional Rastrigin function $J_{c, d}^3(x)$ in \eqref{obj:Rastrigin} with $c = 0.05$ and $0.03$. The first and the second rows are the MSE and N-CP versus iterations, respectively.
    }
    \label{fig:10d}
\end{figure}

\section{Conclusions}
\label{sec:concludsion}
This paper introduces an SGD method with an adaptive Gaussian noise mechanism and proves its global convergence for minimizing nearly convex functions. 
The Gaussian noise added at each step is determined by the difference between the current and optimal values of the objective function. 
Additionally, we propose a double-loop iterative scheme that estimates the lower bound of the objective function without requiring its exact optimal value. 
Empirical results validate the effectiveness of the proposed algorithms. Future work could explore extensions to other stochastic methods and broader classes of objective functions to further generalize these findings.

\begin{appendix}
\section{Proofs of technical lemmas}

\subsection{Proof of Lemma \ref{lm:eps_bound_new}}
\label{sec:proof1}
\begin{proof}
Define the function $\vartheta(x):=\varepsilon(x+x^*)$,
so that $\vartheta(0)=0$ and $ |\vartheta(x) |\le\frac{l}{2}\norm{x}^2$.
Moreover, it holds that
$$
\nabla\varepsilon(x- \delta\xi)
=\nabla\varepsilon(x- \delta\xi-x^*+x^*)
=\nabla \vartheta(x-x^*- \delta\xi)
=\nabla \vartheta(v- \delta\xi)
\quad\mbox{with}\quad v:=x-x^*.
$$
Note that the right-hand side of \eqref{eq:eps_bound.1} is only related to the length of $v=x-x^*$.
Therefore, without loss of generality, we can assume $v= \norm{v}e_1$, where $e_1$ is the unit vector: $e_1 = (1, 0, \cdots, 0)^\top$.
Consequently,
$$
 \abs{\expect{\xi}{\langle \nabla \varepsilon(x - \delta\xi), x - x^* \rangle}}
= \abs{\expect{\xi}{\langle \nabla \vartheta(v- \delta\xi), v \rangle}}
= \abs{\expect{\xi}{\partial_{v_1}\vartheta(v - \delta\xi)}}\cdot\norm{v}.
$$
By denoting $\tilde{\vartheta}(\xi) = \vartheta(v - \delta \xi)$ one has
$\partial_{\xi_1}\tilde{\vartheta}(\xi)=-\delta\partial_{v_1}\vartheta(v - \delta \xi)$.
Note that $\rho(\xi) = \frac{d^{\frac{d}{2}}}{(2\pi)^{\frac{d}{2}}} \exp \kh{-\frac{d\norm{\xi}^2}{2}}$ so that
$\nabla \rho(\xi) = -d\rho(\xi)\xi $.
Therefore, due to partial integration,
$$
\expect{\xi}{\partial_{\xi_1}\tilde{\vartheta}(\xi)}
= \int_{\real^d} \partial_{\xi_1}\tilde{\vartheta}(\xi)\rho(\xi) \dif \xi
= -\int_{\real^d} \tilde{\vartheta}(\xi) \partial_{\xi_1}\rho(\xi) \dif \xi
= d\int_{\real^d} \tilde{\vartheta}(\xi) \rho(\xi)\xi_1 \dif \xi.
$$
Consequently,
$$
\abs{\expect{\xi}{\partial_{\xi_1}\tilde{\vartheta}(\xi)}}
\le d\int_{\real^d} \abs{\tilde{\vartheta}(\xi)} \abs{\xi_1}\rho(\xi) \dif \xi\\
\le \frac{l d}{2}\int_{\real^d} \norm{v - \delta \xi}^2 \abs{\xi_1}\rho(\xi) \dif \xi.
$$
Since $v = \norm{v}e_1$, we have
$\norm{v - \delta \xi}^2 = \norm{v}^2 + \delta^2 \norm{\xi}^2 - 2\delta\norm{v}\xi_1$.
Hence
$$
\abs{\expect{\xi}{\partial_{\xi_1}\tilde{\vartheta}(\xi)}}
\le \frac{l d}{2}\int_{\real^d} \kh{\norm{v}^2 + \delta^2 \norm{\xi}^2} \abs{\xi_1}\rho(\xi) \dif \xi\\
= \frac{l d}{2}\expect{\xi}{\kh{\norm{v}^2 + \delta^2 \norm{\xi}^2} \abs{\xi_1}}.
$$
One can write $\xi = (\xi_1, \tilde{\xi}^\top)^\top$ with $\xi_1 \in \field{N}\kh{0, \frac{1}{d}}, \tilde{\xi} \in \field{N}\kh{0, \frac{1}{d}I_{d - 1}}$.
According to \Cref{lem:intall} and the fact that $\Gamma(z+1)=z \Gamma(z)$,
$$
\ex{ |{\xi_1}|} = \frac{\sqrt{2}}{\sqrt{d}}\frac{\Gamma\kh{1}}{\Gamma\kh{\frac{1}{2}}} =\frac{\sqrt{2}}{\sqrt{d \pi}},\quad
\ex{ |\xi_1|^3} = \frac{2\sqrt{2}}{d^\frac{3}{2} \sqrt{\pi}}
\quad\mbox{and}\quad
\ex{\norm{\tilde\xi}^2} = \frac{ d-1 }{d}.
$$
Therefore,
$$
\begin{aligned}
\abs{\expect{\xi}{\partial_{\xi_1}\tilde{\vartheta}(\xi)}}
&\le \frac{l d}{2}\expect{\xi}{\norm{v}^2\abs{\xi_1} + \delta^2\abs{\xi_1}^3 + \delta^2\norm{{\tilde\xi}}^2\abs{\xi_1}}
\\&= \frac{l d}{2}\norm{v}^2\expect{\xi}{\abs{\xi_1} }
+
\delta^2\frac{l d}{2}\expect{\xi}{\abs{\xi_1}^3}
+
\delta^2\frac{l d}{2}\expect{\xi}{\norm{{\tilde\xi}}^2\abs{\xi_1}}
\\&
= \frac{l d}{2}\norm{v}^2 \frac{\sqrt{2}}{\sqrt{d \pi}}
+
\delta^2\frac{l d}{2}\frac{2\sqrt{2}}{d^\frac{3}{2} \sqrt{\pi}}
+
\delta^2\frac{l d}{2} \frac{ d-1 }{d}\frac{\sqrt{2}}{\sqrt{d \pi}}
\\&
= {l } \norm{v}^2 \frac{\sqrt{ d}}{\sqrt{2 \pi}}
+
\delta^2 {l } \frac{ 2}{ \sqrt{2 d \pi}}
+
\delta^2 {l } \frac{d-1 }{\sqrt{2 d \pi}}
\\&
=
l\frac{\sqrt{ d}}{\sqrt{2 \pi}}
\kh{\norm{v}^2 + \delta^2\kh{1 + \frac{1}{d}}}.
\end{aligned}
$$
Consequently, one has
$$
\abs{\expect{\xi}{\partial_{v_1}\vartheta(v - \delta\xi)}}\cdot\norm{v}
\le \frac{1}{\delta } \abs{\expect{\xi}{\partial_{\xi_1}\tilde{\vartheta}(\xi)}} \cdot\norm{v}
\le l\frac{\sqrt{ d} \norm{v}}{\delta \sqrt{2 \pi}}
\kh{\norm{v}^2 + \delta^2\kh{1 + \frac{1}{d}}},
$$
so that \eqref{eq:eps_bound.1} holds.

Recall that $\tilde{\vartheta}(\xi) = \vartheta(v - \delta \xi)$ and
$\partial_{\xi_j}\tilde{\vartheta}(\xi)=-\delta\partial_{v_j}\vartheta(v - \delta \xi),\forall j=1,\ldots, d$.
Thus:
$$
\langle \nabla \vartheta(v- \delta\xi), \delta \xi \rangle
 = \sum_{j = 1}^d \partial_{v_j}\vartheta(v - \delta \xi) \delta \xi_j
 = -\sum_{j = 1}^d \partial_{\xi_j}\tilde{\vartheta}(\xi) \xi_j.
$$
Note that $\rho(\xi) = \frac{d^{\frac{d}{2}}}{(2\pi)^{\frac{d}{2}}} \exp \kh{-\frac{d\norm{\xi}^2}{2}}$ so that
$\nabla \rho(\xi) = -d\rho(\xi)\xi $.
Consequently, by partial integration,
$$
\expect{\xi}{\langle \nabla \vartheta(v- \delta\xi), \delta \xi \rangle}
= -\sum_{j = 1}^d
 \int_{\real^d}
 \partial_{\xi_j}\tilde{\vartheta}(\xi) \xi_j
 \rho(\xi) \dif \xi= \sum_{j = 1}^d
 \int_{\real^d} \tilde{\vartheta}(\xi) \partial_{\xi_j} (\xi_j\rho(\xi)) \dif\xi
= \int_{\real^d} \tilde{\vartheta}(\xi) \divi(\rho(\xi)\xi) \dif \xi.
$$
Since $\nabla \rho(\xi) = -d\rho(\xi)\xi $, $\divi(\rho(\xi)\xi) = \langle \nabla \rho(\xi), \xi \rangle + d \rho(\xi) = d\rho(\xi)\kh{1 - \norm{\xi}^2}$.
Therefore,
$$
\begin{aligned}
\abs{\expect{\xi}{\langle \nabla \vartheta(v- \delta\xi), \delta \xi \rangle}}
&\le \frac{l}{2}\int_{\real^d} \norm{v - \delta \xi}^2 \abs{\divi(\rho(\xi)\xi)} \dif \xi
= \frac{dl}{2}\ex{\norm{v - \delta\xi}^2\abs{1 - \norm{\xi}^2}}
\\
 &= \frac{dl}{2}\ex{\kh{\norm{v}^2 + \delta^2\norm{\xi}^2}\abs{1 - \norm{\xi}^2}}.
 \end{aligned}
$$

According to \eqref{eq:int.1234} of \Cref{lem:intall}, we have
$$
\ex{\abs{1 - \norm{\xi}^2}} \le \sqrt{\ex{\kh{1 - \norm{\xi}^2}^2}}
= \sqrt{\ex{1 - 2\norm{\xi}^2 + \norm{\xi}^4}}
= \sqrt{\frac{2}{d}},
$$
and
$$
\ex{\abs{1 - \norm{\xi}^2}\norm{\xi}^2}
\le \sqrt{\ex{\kh{1 - \norm{\xi}^2}^2}}\sqrt{\ex{\norm{\xi}^4}}
= \sqrt{\frac{2}{d}}\sqrt{1 + \frac{2}{d}} \le \sqrt{\frac{2}{d}}\kh{1 + \frac{1}{d}}.
$$
Therefore,
$$
\abs{\expect{\xi}{\langle \nabla \vartheta(v- \delta\xi), \delta \xi \rangle}}
\le l \sqrt{\frac{d}{2}}\norm{v}^2 + l \delta^2 \sqrt{\frac{d}{2}}\kh{1 + \frac{1}{d}}\\
= l \sqrt{\frac{d}{2}}
\kh{\norm{v}^2 + \delta^2\kh{1 + \frac{1}{d}}},
$$
which implies \eqref{eq:eps_bound.2}.
Summing \eqref{eq:eps_bound.1} and \eqref{eq:eps_bound.2} up, one can readily get
$$
\abs{\expect{\xi}{\langle \nabla \varepsilon(x - \delta\xi), x - \delta\xi - x^* \rangle}}
\le l \sqrt{\frac{d}{2}} \kh{\frac{\norm{x - x^*}}{\delta\sqrt{\pi}} + 1}
\kh{\norm{x - x^*}^2 + \delta^2\kh{1 + \frac{1}{d}}},
$$
which implies \eqref{result:eps_bound}, and this completes the proof.
\end{proof}

\subsection{Proof of Lemma \ref{lm:st}}
\label{sec:proof2}
\begin{proof}
Define the stopping time
$$
\tau : = \inf\{t\mid X_t<\ell\}.
$$
Note if $\{t\mid X_t<\ell\} = \emptyset$, we define $\tau = +\infty$. Obviously, $\{\tau>t\}$ is $\field{F}_t$-measurable, so that
$$
 \ex{X_{t + 1}\mathds{1}\{\tau > t\} \mid \field{F}_t} \le \theta X_t\mathds{1}\{\tau > t\}.
$$
By taking expectation with respect to $X_0,\ldots, X_t$ on both sides one can get
 \begin{equation}\label{eq:lm.st.1}
 \ex{X_{t + 1}\mathds{1}\{\tau > t\}} \le \theta \ex{X_t\mathds{1}\{\tau > t\}}.
 \end{equation}
Note that $\mathds{1}\{\tau > t\} = \mathds{1}\{\tau > t + 1\} + \mathds{1}\{\tau = t + 1\}$, so that
 \begin{align}\label{eq:lm.st.2}
 \ex{X_{t + 1}\mathds{1}\{\tau > t\}}
 &= \ex{X_{t + 1}\mathds{1}\{\tau > t + 1\}} + \ex{X_{t + 1}\mathds{1}\{\tau = t + 1\}}\notag \\
 &\ge \ex{X_{t + 1}\mathds{1}\{\tau > t + 1\}} - b\proba{\tau = t + 1}.
 \end{align}
Combining \eqref{eq:lm.st.1} and \eqref{eq:lm.st.2} together obtains
$$
 \ex{X_{t + 1}\mathds{1}\{\tau > t + 1\}} \le \theta \ex{X_t\mathds{1}\{\tau > t\}} + b\proba{\tau = t + 1}.
$$
Then we have
 \begin{equation}\label{eq:lm.st.eq1}
 \ex{X_{t + 1}\mathds{1}\{\tau > t + 1\}} \le \theta^{t + 1} \ex{X_0\mathds{1}\{\tau > 0\}}
 + b \sum_{k = 1}^{t + 1} \theta^{t + 1 - k}\proba{\tau = k}.
 \end{equation}
Note that
\[\label{eq:lm.st.5}
\begin{aligned}
 \sum_{k = 1}^{t + 1} \theta^{t + 1 - k}\proba{\tau = k}
 &= \sum_{k = 1}^{t + 1} \theta^{t + 1 - k}\kh{\proba{\tau > k - 1} - \proba{\tau > k}} \\
 &= \sum_{k = 0}^{t} \theta^{t - k}\proba{\tau > k}
 - \sum_{k = 1}^{t + 1} \theta^{t + 1 - k}\proba{\tau > k} \\
 &= (1 - \theta)\sum_{k = 0}^{t} \theta^{t - k}\proba{\tau > k} + \theta^{t + 1}\proba{\tau > 0} - \proba{\tau > t + 1}.
 \end{aligned}
\]
Moreover, according to \eqref{def:st.B} one has $B = \ex{X_0 \mathds{1}\{X_0 \ge \ell\}} =\ex{X_0\mathds{1}\{\tau > 0\}} \ge \ex{\ell\mathds{1}\{\tau > 0\}} = \ell\proba{\tau > 0}$, so that
 \begin{equation}\label{eq:lm.st.6}
 \proba{\tau > 0} \le \dfrac{B}{\ell}.
 \end{equation}
Therefore, by substituting \eqref{eq:lm.st.5} and \eqref{eq:lm.st.6} into \eqref{eq:lm.st.eq1}, we have
\begin{equation}\label{eq:lm.st.7}
 \begin{aligned}
\ex{X_{t + 1}\mathds{1}\{\tau > t + 1\}} &\le \theta^{t + 1} B
 + b \kh{(1 - \theta)\sum_{k = 0}^{t} \theta^{t - k}\proba{\tau > k} + \theta^{t + 1}\dfrac{B}{\ell} - \proba{\tau > t + 1}}\\
&\le \theta^{t + 1}\kh{1 + \frac{b}{\ell}}B - b\proba{\tau > t + 1}+ (1 - \theta)b\sum_{k = 0}^{t} \theta^{t - k}\proba{\tau > k}.
 \end{aligned}
 \end{equation}
On the other hand, it holds that $\ex{X_{t + 1}\mathds{1}\{\tau > t + 1\}} \ge \ell\proba{\tau > t + 1}$, which together with \eqref{eq:lm.st.7} implies
 \begin{equation}\label{eq:lm.st.main}
 \proba{\tau > t + 1} \le \theta^{t + 1}\frac{B}{\ell} +
 \frac{(1 - \theta)b}{b + \ell}\sum_{k = 0}^{t} \theta^{t - k}\proba{\tau > k}.
 \end{equation}
In the following, we prove by conduction that
 \begin{equation}\label{eq:lm.st.main_result}
 \proba{\tau > t} \le \kh{\frac{b + \theta\ell}{b + \ell}}^t\frac{B}{\ell},\quad\forall 0\le t \le M.
 \end{equation}
Note that \eqref{eq:lm.st.main_result} trivially holds for $t = 0$ according to \eqref{eq:lm.st.6}.
Now, assume that \eqref{eq:lm.st.main_result} is verified for all $t\le n$ with $n\ge 0$.
Then by using \eqref{eq:lm.st.main} one can get
\begin{equation}\label{eq:lm.st.9}
 \proba{\tau > n + 1} \le \theta^{n + 1}\frac{B}{\ell} +
 \frac{(1 - \theta)b}{b + \ell}\sum_{t = 0}^{n} \theta^{n - t}\kh{\frac{b + \theta\ell}{b + \ell}}^t\frac{B}{\ell}.
\end{equation}
Note that
$$
\sum_{t = 0}^{n} \kh{\frac{b + \theta\ell}{\theta b + \theta\ell}}^t
= \frac{\kh{\frac{b + \theta\ell}{\theta b + \theta\ell}}^{n + 1} - 1}
{\kh{\frac{b + \theta\ell}{\theta b + \theta\ell}} - 1}
= \frac{\theta b + \theta\ell}{(1 - \theta)b}
\kh{\kh{\frac{b + \theta\ell}{\theta b + \theta\ell}}^{n + 1} - 1}.
$$
Thus we can get from \eqref{eq:lm.st.9} that
$$\begin{aligned}
\proba{\tau > n + 1} &\le \theta^{n + 1}\frac{B}{\ell} +
\frac{(1 - \theta)b}{b + \ell} \theta^t\sum_{t = 0}^{n}
\kh{\frac{b + \theta\ell}{\theta b + \theta\ell}}^t\frac{B}{\ell}\notag \\
&= \theta^{t + 1}\frac{B}{\ell} + \theta^{t + 1}
\kh{\kh{\frac{b + \theta\ell}{\theta b + \theta\ell}}^{n + 1} - 1}\frac{B}{\ell}
= \kh{\frac{b + \theta\ell}{b + \ell}}^{n + 1}\frac{B}{\ell}.
 \end{aligned}
 $$
Consequently, \eqref{eq:lm.st.main_result} holds. Therefore,
$$
 \proba{\tau \le M} \ge 1 - \kh{\frac{b + \theta\ell}{b + \ell}}^M\frac{B}{\ell}.
$$
Since the event $\tau \le M$ is equivalent to the existence of $ t \le M$ such that $X_t < \ell$, the Lemma is proved.
\end{proof}

\section{A Comparison of regularity conditions}
\label{appendix:reg_cond.sin_log}
Here we make a comparison of the regularity considerations introduced in \Cref{sec:ncf}
by imposing them on Example \ref{example:sin_log}. 
Recall that, for the given $\varepsilon \in (0, 1)$ and $R > 0$, the function $J_{\varepsilon, R}^2(x)$ in \cref{example:sin_log} is given by
\begin{equation*}
J_{\varepsilon, R}^2(x) = 
\begin{cases}
    \dfrac{1 + \varepsilon\sin\kh{2R\log \abs{x}}}{2}x^2, & x \ne 0,\\
    0, & x = 0.
\end{cases}
\end{equation*}
It is clear that $J_{\varepsilon, R}^2(x)$ is an even and continuously differentiable function on $\real$. Since $\frac{1 - \varepsilon}{2}x^2 \le J_{\varepsilon, R}^2(x) \le \frac{1 + \varepsilon}{2}x^2$, we have that $x^* = 0$ is the unique global minimizer with $f^* = 0$, and $J_{\varepsilon, R}^2(x)$ satisfies the QG condition \eqref{eq:QG} with parameter $\mu_q = 1 - \varepsilon$. Moreover, for all $x > 0$, we have
\begin{equation*}
    (J_{\varepsilon, R}^2)'(x) = \brc{1 + \varepsilon\sin(2R\log x) + \varepsilon R\cos(2R\log x)}x,
\end{equation*}
\begin{equation*}
    (J_{\varepsilon, R}^2)''(x) = 1 + \varepsilon(1 - 2R^2)\sin(2R\log x) + 3\varepsilon R\cos(2R\log x).
\end{equation*}
Hence, by the Cauchy-Schwarz inequality, we get
\begin{equation*}
    (J_{\varepsilon, R}^2)''(x) \ge 1 - \varepsilon\sqrt{(1 - 2R^2)^2 + 9R^2} = 1 - \varepsilon\sqrt{1 + 5R^2 + 4R^4}.
\end{equation*}
When $\varepsilon\sqrt{1 + 5R^2 + 4R^4} < 1$, it follows from the above inequality that $(J_{\varepsilon, R}^2)''(x) > 0$, and $J_{\varepsilon, R}^2(x)$ is strongly convex with parameter $1 - \varepsilon\sqrt{1 + 5R^2 + 4R^4}$. 
Similarly, we have
\begin{equation*}
    1 + \varepsilon\sin(2R\log x) + \varepsilon R\cos(2R\log x) \ge 1 - \varepsilon\sqrt{1 + R^2}.
\end{equation*}
Therefore, when $\varepsilon\sqrt{1 + R^2} < 1$, we obtain
\begin{equation*}
    (J_{\varepsilon, R}^2)'(x)(x - x^*) \ge \brc{1 - \varepsilon\sqrt{1 + R^2}}\abs{x - x^*}^2,
\end{equation*}
\begin{equation*}
    \abs{(J_{\varepsilon, R}^2)'(x)}^2 \ge \brc{1 - \varepsilon\sqrt{1 + R^2}}^2x^2 \ge \frac{2\brc{1 - \varepsilon\sqrt{1 + R^2}}^2}{1 + \varepsilon}(J_{\varepsilon, R}^2(x) - f^*).
\end{equation*}
These inequalities imply that, under this condition, $J_{\varepsilon, R}^2(x)$ satisfies RSI \eqref{eq:RSI} with parameter $\mu_r = 1 - \varepsilon\sqrt{1 + R^2}$, and PL condition \eqref{eq:PL} with parameter $\mu_p = \frac{\brc{1 - \varepsilon\sqrt{1 + R^2}}^2}{1 + \varepsilon}$.

\end{appendix}

\section*{Declarations}

{\bf Conflict of interest}
The authors have no relevant financial or non-financial interests to disclose.

\bibliographystyle{plain}
\bibliography{Ref}
\end{document}